\documentclass[reqno]{amsart}
\usepackage[T1]{fontenc}
\usepackage{latexsym,amssymb,amsmath,xcolor}

\usepackage[all,ps]{xy}
\usepackage{amsthm}
\usepackage{longtable}
\usepackage{epsfig}
\usepackage{mathrsfs}
\usepackage{hhline}
\usepackage{epic}
\usepackage{enumerate}
\usepackage{xcolor}
\usepackage{pgf,tikz}
\usepackage{mathrsfs}
\usetikzlibrary{arrows,shapes,calc}

 \newcommand{\Alt}{\mathfrak{A}}

 \newcommand{\sym}{\mathfrak{S}}
 \newcommand{\alt}{\mathfrak{A}}

 \newcommand{\Ind}{\operatorname{Ind}}
 \newcommand{\Res}{\operatorname{Res}}

\newcommand{\sgn}{\operatorname{sgn}}

\newcommand{\legendre}[2]{\left(\frac{#1}{#2}\right)}

 \newcommand{\B}{\widetilde{B}}
 \newcommand{\bs}{\widetilde{b}}

\newcommand{\N}{\mathbb{N}}
 \newcommand{\Q}{\mathbb{Q}}
 \newcommand{\Z}{\mathbb{Z}}

 \newcommand{\Sym}{\mathfrak{S}}

 \newcommand{\Irr}{\operatorname{Irr}}

\newcommand{\tSym}{\widetilde{S}}
\newcommand{\tAlt}{\widetilde{A}}

\newdimen\shadedBaseline\shadedBaseline=-4mm
\newcount\tableauRow\newcount\tableauCol
\newcommand\ShadedTableau[2][\relax]{%
  \begin{tikzpicture}[scale=0.4,draw/.append style={thick,black},baseline=\shadedBaseline]
    \ifx\relax#1\relax%
    \else 
      \foreach\bx in {#1} { \filldraw[blue!20]\bx+(-.5,-.5)rectangle++(.5,.5); }
    \fi
    \tableauRow=0
    \foreach \Row in {#2} {
       \tableauCol=1
       \foreach\k in \Row {
          \draw(\the\tableauCol,\the\tableauRow)+(-.5,-.5)rectangle++(.5,.5);
          \draw(\the\tableauCol,\the\tableauRow)node{\k};
          \global\advance\tableauCol by 1
       }
       \global\advance\tableauRow by -1
    }
  \end{tikzpicture}%
}
\newcommand\diag[3][\relax]{%
  \begin{tikzpicture}[scale=0.4,draw/.append style={thick,black},baseline=\shadedBaseline]
    \ifx\relax#1\relax%
    \else 
      \foreach\bx in {#1} { \filldraw[blue!20,dashed]\bx+(-.5,-.5)rectangle++(.5,.5); }
    \fi
    \tableauRow=0
    \foreach \Row in {#2} {
       \tableauCol=1
       \foreach\k in \Row {
          \draw(\the\tableauCol,\the\tableauRow)+(-.5,-.5)rectangle++(.5,.5);
          \draw(\the\tableauCol,\the\tableauRow)node{\k};
          \global\advance\tableauCol by 1
       }
       \global\advance\tableauRow by -1
    }
    \foreach \x in {#3} {
      \draw[red,dashed](\x+3,-2-\x)+(-.5,-.5)rectangle++(.5,.5);
    }
  \end{tikzpicture}%
}

\newcommand\frob[7][\relax]{%
  \begin{tikzpicture}[scale=0.4,draw/.append style={black},baseline=\shadedBaseline]
    \ifx\relax#1\relax%
    \else 
      \foreach\bx in {#1} { \filldraw[gray!20,dashed]\bx+(-.5,-.5)rectangle++(.5,.5); }
    \fi
    \tableauRow=0
    \foreach \Row in {#2} {
       \tableauCol=1
       \foreach\k in \Row {
          \draw(\the\tableauCol,\the\tableauRow)+(-.5,-.5)rectangle++(.5,.5);
          \draw(\the\tableauCol,\the\tableauRow)node{\k};
          \global\advance\tableauCol by 1
       }
       \global\advance\tableauRow by -1
    }
        \foreach \x in {#3} {
      \draw[gray!80,dashed](\x+3,-2-\x)+(-.5,-.5)rectangle++(.5,.5);
    }
    \foreach\bx in {#4} {
        \draw[thick]\bx+(-.3,0)rectangle++(.3,.0);
    }
    \foreach\bx in {#5} {
        \draw[thick]\bx+(0,.3)rectangle++(0,-.3);
    }
    \foreach\bx in {#6} {
        \draw[thick,densely dotted]\bx+(-.3,0)--++(.3,.0);
    }
    \foreach\bx in {#7} {
        \draw[thick,densely dotted]\bx+(0,.3)--++(0,-.3);
    }

  \end{tikzpicture}%
}

\newtheorem{theorem}{Theorem}[section] 
\newtheorem{lemma}[theorem]{Lemma}     
\newtheorem{corollary}[theorem]{Corollary}
\newtheorem{proposition}[theorem]{Proposition}

\newtheorem{notation}[theorem]{Notation}
\newtheorem{example}[theorem]{Example}
\theoremstyle{definition}
\newtheorem{remark}[theorem]{Remark}

\title[]
{Galois Automorphisms And Littlewood Decompositions}

\author{Olivier Brunat}

\address{Universit\'e Paris Cit\'e\\ Institut de math\'ematiques de
         Jussieu -- Paris Rive Gauche\\ UFR de math\'e\-matiques\\ Case
7012\\ 75205 Paris Cedex 13\\
         France.}
\email{olivier.brunat@imj-prg.fr}

\author{Rishi Nath}

\address{York College, City University of New York, 
94--20 Guy R. Brewer Blvd. \\
Jamaica, NY 11435\\
USA
}
\email{rnath@york.cuny.edu}

\subjclass[2010]{Primary 20C30; Secondary 20C15}

\begin{document} 

\begin{abstract} 
    The study of modular representation theory of the double covering
groups of the symmetric and alternating groups reveals rich and subtle
combinatorial and algebraic phenomena involving their irreducible
characters and the structure of their $p$-blocks, where $p$ is an odd
prime number. In this paper, we investigate the action of certain Galois
automorphisms, those that act on $p'$-roots of unity by a power of $p$, on
spin characters, with an emphasis on their interaction with perfect
isometries and block theory. In particular, we prove that perfect
isometries constructed by the first author and J.\,B. Gramain in
\cite{BrGr3}, which were used to establish a weaker form of the
Kessar--Schaps conjecture, remain preserved under this Galois action
whenever certain natural compatibility conditions occur.
\end{abstract} 
\maketitle

\section{Introduction}

    The interplay between Galois theory and modular representation theory
has long been a source of deep and intriguing questions. A particularly
influential viewpoint in this area was introduced by G. Navarro in his
\emph{Galois refinement} of the McKay conjecture~\cite{NavarroGalois}.
While the original conjecture, which concerns characters of degree prime
to a fixed prime $p$,  predicts a bijection between such characters of a
finite group and those of the normalizer of a Sylow $p$-subgroup (a result
now proven, following nearly twenty years of progress built on the
foundational work of Isaacs, Malle, and Navarro \cite{IMN}, and finalized
by the recent contribution of Cabanes and Späth \cite{CabSpath}), Navarro
proposed that such bijections could in fact be chosen to be equivariant
under the action of a certain subgroup $\mathcal H$ of the absolute Galois
group $\operatorname{Gal}(\overline{\mathbb{Q}}/\mathbb{Q})$. This
subgroup acts on $p'$-roots of unity, that is a complex root of unity
whose order is prime to $p$, via powers of~$p$. This refined conjecture,
now known as the \emph{Galois-McKay conjecture}, has since been the
subject of extensive research (see for example~\cite{team}). In
particular, it was reduced to a statement about finite quasi-simple groups
by Navarro–Späth–Vallejo in~\cite{NavarroSpaethVallejo}, a key step that
has guided much of the recent progress on the topic.

    Inspired by the general principle that character correspondences
should encode the action of Galois automorphisms, we extend the
McKay–Navarro program to a more general setting, namely, to bijections
between natural subsets of characters (not just $p'$-characters or
height-zero characters in $p$-blocks), and, in particular we focus on the
double covering groups of the symmetric and alternating groups for
\emph{an odd prime $p$}. 
    These central extensions give rise to \emph{spin characters}, that is,
faithful irreducible characters of the covering groups, whose values
reflect subtle arithmetic and combinatorial features, notably partitions
with distinct parts (called \emph{bar-partitions}) and specific roots of
unity. This makes their behavior under Galois automorphisms especially
intricate to analyze.
\medskip

    To study the action of $\mathcal H$ on spin characters of the double
covering groups, we construct a character correspondence involving a
product-type group
\begin{equation}
\label{eq:Gintro}
G=\tSym_r \, \widehat{\times} \, \tSym_{pw}.
\end{equation}
    The parameters $r$ and $w$ arise naturally from the combinatorics of
bar-partitions, namely for such a partition $\lambda$ labelling one or two
irreducible spin characters, $r$ is the size of its $p$-bar core, and $w$
is determined from its $p$-bar quotient.
    This setup, a special case of the twisted products studied by
Humphreys~\cite{Humphreys}, provides an algebraic framework for the
$p$-bar Littlewood decomposition, which decomposes a bar-partition into
$p$-bar core and $p$-bar quotient, thereby reducing the analysis of spin
characters to simpler components. More precisely, this interpretation
connects the \emph{combinatorial level}, that is the decomposition of
bar-partitions into $p$-bar-core and $p$-bar-quotient, and the
\emph{algebraic level}, where the representation theory of $G$ reflects
this structure. This perspective not only clarifies the conceptual
approach and provides practical tools for the analysis of
character-theoretic phenomena at a fixed prime $p$, but also sheds new
light on earlier constructions by placing them within a unified framework.

    In this paper, we illustrate this principle through a concrete
example, culminating in one of our main results: the correspondence we
construct is $\mathcal H$-equivariant. In other words, the action of
$\mathcal H$ on spin characters respects the combinatorial structure
encoded in the $p$-bar Littlewood decomposition. This compatibility allows
us to translate the action of $\mathcal H$ into more accessible
combinatorial terms.

\smallskip
    Extending this perspective to the level of $p$-blocks, we recall that
these play a central role in modular representation theory of finite
groups, organizing irreducible characters according to their modular
behavior at the prime $p$. In the context of the Galois-McKay program, it
is natural to consider the action of $\mathcal H$ not only on individual
characters, but also on entire $p$-blocks. Given two $p$-blocks of
(possibly distinct) finite groups, we say they are \emph{Navarro-Galois
equivalent} if there exists a bijection between their sets of irreducible
characters that is equivariant under the action of $\mathcal H$.

    Building on the correspondence introduced earlier, one of the main
results of this paper is the construction of a bijection between the spin
$p$-blocks of the double cover of the symmetric and alternating group, and
certain $p$-blocks of the twisted product group $G$ introduced in
Equation~(\ref{eq:Gintro}) (or its alternating counterpart). Here, $r$
denotes the size of the $p$-bar core labeling the spin block, and $w$ is
its $p$-bar weight. 

    We show that this bijection is $\mathcal H$-equivariant and preserves
important block-theoretic invariants, such as defect groups and character
$p$-heights (see Theorem~\ref{thm:bijblockstildehump}). Crucially, this
lifts the reduction principle introduced above to a deeper structural
level, showing that the modular representation theory of spin blocks can
be studied more effectively within the accessible framework provided by
$G$.
    Beyond its implications for the Galois action on character values,
this bijection is of independent interest as a significant structural
result in the representation theory of spin blocks. 
\par\smallskip
    As an application of our results, we deduce that various perfect
isometries between $p$-blocks of spin characters, constructed in earlier
work, are in fact compatible with the action of $\mathcal H$ (see
Remarks~\ref{rk:perfectdirect}, \ref{rk:perfectcross} and
\ref{rk:isoalt}).

    In 1990, within the framework of Brou\'e's conjecture~\cite{Broue},
Michel Enguehard proved~\cite{Enguehard} that two non-spin $p$-blocks of
$\tSym_n$ and $\tSym_m$ are perfectly isometric if and only if they share
the same $p$-weight. Later, J. Chuang and R. Rouquier extended this to
derived equivalences through $\mathfrak{sl}_2$-categorification. More
recently, the first author and J.\,B. Gramain generalized in~\cite{BrGr3}
Enguehard's results to alternating groups and spin blocks of double
covers, constructing perfect isometries that in particular support a
weaker version of a conjecture stated by R. Kessar and M. Schaps
in~\cite{Kessar-Schaps}.

    A key result of this paper is that the perfect isometries under
consideration are $\mathcal{H}$-equivariant, assuming certain natural
compatibility conditions on the cores.

    Note that the bijections we construct in this paper are compatible
with Galois actions and preserve not only the $p'$-degrees but also the
$p$-heights of all characters, not just those of height zero as in the
classical McKay conjecture and its blockwise refinements. Although our
results are established in the symmetric and spin setting, they represent
a promising generalization both refining our understanding of the
interplay between Galois actions and block theory, and suggests further
directions in the study of $\mathcal{H}$-equivariant correspondences.
\vspace{0.42em}\par
    The article is organized as follows. In Section~\ref{sec:part1}, we
introduce the basic setting and notation used throughout the paper.
Section~\ref{sec:part2} recalls the $p$-bar Littlewood decomposition of
bar-partitions. While we refer to~\cite{olsson} for details, our
presentation follows the approach of~\cite{BrNa2}, using the framework of
the \emph{pointed abacus}.
    A key combinatorial result is Lemma~\ref{lemma:longueurlittelwoodec},
which relates the number of parts of a partition to those of its $p$-bar
Littlewood components, and connects closely with the push-and-pull
procedure from the theory developed in~\cite{BrNa2}.

    Section~\ref{sec:actiongalois} focuses on recalling the values of
irreducible spin characters of $\tSym_n$ and $\tAlt_n$, setting the stage
to describe the action of Galois-Navarro automorphisms on them. The main
result (Theorem~\ref{thm:little}) allows us to reduce the study of this
action to spin characters labeled by $p$-bar cores and \emph{$p$-bar
cocore partitions}, that is, partitions with empty $p$-bar cores. The
proof relies heavily on combinatorial tools such as the push-and-pull
process and the fundamental notion of \emph{bar-abacus pairing}. Notably,
the statement involves some subtle features, including a seemingly
unexpected sign factor, whose explanation is deferred to later sections.

    In Section~\ref{sec:humphreys}, we introduce the spin representations
of the Humphreys twisted product groups $\widetilde{G}_1 \widehat{\times}
\widetilde{G}_2$, where $\widetilde{G}_1$ and $\widetilde{G}_2$ belong to
a specific class $\mathcal{G}$. Following Humphreys~\cite{Humphreys}, we
provide explicit character values for the irreducible spin representations
of these twisted products as well as for their index-two subgroups.

    Building on this framework, in Section~\ref{sec:main}, we analyze the
group $G = \tSym_r \widehat{\times} \tSym_{pw}$ and establish a
correspondence that associates to each irreducible spin character of
$\tSym_n$ (labeled by a bar-partition $\lambda$) a spin character of $G$,
and similarly for $\tAlt_n$ with its corresponding ``alternating
subgroup'' $G^+$. This construction clarifies the behavior of spin
characters and their Galois action by expressing them in terms of their
decomposed components. More precisely, Theorem~\ref{thm:main} proves that
this correspondence is $\mathcal H$-equivariant.

    In Section~\ref{sec:pblocks}, we focus on the structure of spin
$p$-blocks of the double covers $\tSym_n$ and $\tAlt_n$. Recall that such
a spin block is naturally parameterized by a $p$-bar core $\kappa$ and an
integer $w$, called its $p$-bar weight. 
In Theorem~\ref{thm:blockhat}, we show that the groups $G$ and $G^+$
introduced earlier also possess $p$-blocks naturally parameterized by the
same data $(\kappa, w)$, and we describe important invariants of these
blocks, such as their defect groups and the $p$-height of their
characters.  
Then, Theorem~\ref{thm:bijblockstildehump} establishes a bijection between
the spin $p$-block of $\tSym_n$ (resp. $\tAlt_n$) labeled by $\kappa$ and
$w$ and the corresponding $p$-block of $G$ (resp. $G^+$) with the same
parameters. This bijection preserves key invariants, including defect
groups and character $p$-heights, and is moreover
$\mathcal{H}$-equivariant.

    From Theorem~\ref{thm:bijblockstildehump}, we deduce several important
consequences regarding Navarro-Galois equivalences of spin and non-spin
$p$-blocks:

\begin{itemize}

\item[$\star$] Two spin $p$-blocks of the double covers of symmetric or
alternating groups, with the same weight and sign, and labeled by $p$-bar
cores satisfying the same rationality condition, are Navarro-Galois
equivalent (see Theorem~\ref{thm:blocksmemesigne}). More precisely, the
perfect isometry constructed in~\cite{BrGr3} between such blocks is
$p$-height and $p$-defect preserving, and $\mathcal H$-equivariant; see
also Remark~\ref{rk:perfectdirect}. 

\item[$\star$] Any two spin $p$-blocks of the double covers, one from the
symmetric group with negative sign and the other from the alternating
group with positive sign, having the same weight and satisfying the same
rationality condition are Navarro-Galois equivalent (see
Theorem~\ref{thm:crossing}). Similarly to the non-crossover case, the
perfect isometry constructed in \cite{BrGr3} preserves defect, $p$-height,
and is $\mathcal{H}$-equivariant, as noted in
Remark~\ref{rk:perfectcross}.
\end{itemize}

    However, as shown in Remark~\ref{rk:fails}, the remaining crossover
case, when the sign of the spin $p$-block of $\tSym_n$ is positive and
that of $\tAlt_n$ is negative, is not covered by
Theorem~\ref{thm:crossing}. 
    In this situation, the two blocks, although related by a perfect
isometry, are \emph{not} $\mathcal{H}$-equivariant. Understanding the
precise reasons behind this failure of $\mathcal{H}$-equivariance would be
an interesting direction for future research, as it likely stems from
intrinsic structural differences in the blocks themselves.

    Finally, although this case lie outside the spin setting, our approach
still applies. By using the non-spin characters of the groups \(G\) and
\(G^+\) and following the same strategy, we establish in
\S\ref{subsec:nonspin} an analogous result in this context: we prove that
two non-spin \(p\)-blocks of the double covers of alternating groups,
labeled by self-conjugate \(p\)-cores, and having the same weight and the
same rationality condition, are Navarro-Galois equivalent; see
Theorem~\ref{thm:nonspin} and Remark~\ref{rk:isoalt}.

\section{Notation and preliminaries} 
\label{sec:part1}
    For any group finite $G$, as usually, we denote by $\Irr(G)$ the set
of complex irreducible characters of $G$.

\subsection{Groups with index two subgroup}
\label{subsec:index2}

Let $G$ be a finite group and $G^+$ be a subgroup of $G$ of index $2$.
Then, there exists a surjective group homomorphism
\begin{equation}
\label{eq:defsignature}
\varepsilon_{G}:G\longrightarrow \{-1,1\}
\end{equation}
with kernel $\ker(\varepsilon_G)=G^+$. 
    In this setting, two natural group actions arise. First,
$\langle\varepsilon_G\rangle$ acts on $\Irr(G)$ by tensoring with the
one-dimensional character $\varepsilon_G$. Second, the quotient $G/G^+$
induces an action on $\Irr(G^+)$ by conjugation. 

    Clifford theory~\cite[Chapter 6]{isaacs} provides a correspondence
between the orbits of these two actions. More precisely:
\begin{itemize}
\item[$\star$] An $\langle \varepsilon_G\rangle$-orbit
$\{\chi,\varepsilon_G\otimes \chi\}$ of size $2$ in $\Irr(G)$ corresponds
to a  $G/G^+$-orbit $\{\psi\}$ of size $1$ in $\Irr(G^+)$, with the link
given by
$$\psi=\Res_{G^+}^G(\chi)=\Res_{G^+}^G(\varepsilon_G\otimes \chi)\quad
\text{and}\quad \Ind_{G^+}^G(\psi)=\chi+\varepsilon_G\otimes \chi.$$
\item[$\star$] An $\langle\varepsilon_G\rangle$-orbit $\{\chi\}$ of size
$1$ in $\Irr(G)$ corresponds to a $G/G^+$-orbit $\{\psi,{}^g\psi\}$ of
size $2$ in $\Irr(G^+)$, where $g\in G\backslash G^+$, with
$$\Res_{G^+}^G(\chi)=\psi+{}^g\psi\quad\text{and}\quad
\chi=\Ind_{G^+}^G(\psi)=\Ind_{G^+}^G({}^g\psi).$$
\end{itemize}

    Assume that $\Lambda$ is a set parametrizing the $\langle
\varepsilon_G\rangle$-orbits in $\Irr(G)$, or equivalently, the
$G/G^+$-orbits in $\Irr(G^+)$. For $\lambda\in \Lambda$, we denote by
$$\{\chi_{\lambda}\}\quad\text{and}\quad\{\chi_{\lambda}^+,\chi_{\lambda}^-\}$$
the corresponding orbits.

\begin{remark}
\label{rk:notambigu}
    Note that the notation may be ambiguous, as $\chi_{\lambda}$, or
$\chi_{\lambda}^{\pm}$ may denote characters in either $\Irr(G)$ or
$\Irr(G^+)$, depending on the context.
\end{remark}

    In the following, a character lying in an orbit of size $1$ is said to
be \emph{self-associate}, whereas two characters forming an orbit of size
$2$ are said to be \emph{associate}. Such characters are also referred to
as \emph{non-self-associate} characters. 

    For any $\lambda\in\Lambda$, we also define the corresponding
\emph{difference character}, a class function either on $G$, or on $G^+$,
in the following way. If $\lambda$ labels a spin orbit of size $1$ in
$\Irr(G)$, then it labels an orbit of size $2$ in $\Irr(G^+)$. Otherwise,
it labels a spin orbit of size $2$ in $\Irr(G)$. In either case, we define
\begin{equation}
\label{eq:chardiff}
\Delta_{\lambda}=\chi_{\lambda}^+-\chi_{\lambda}^-.
\end{equation}

\subsection{Central extensions}
    Following Humphreys~\cite{Humphreys}, we consider the class $\mathcal
G$ of finite groups $\widetilde G$ equipped with a central involution $z$
and an index-two subgroup $\widetilde{G}^+$. We define
\begin{equation}
\label{eq:stilde}
s_{\widetilde G}:\widetilde{G}\longrightarrow \Z/2\Z
\end{equation}
to be the group homomorphism with kernel $\widetilde G^+$. In particular,
$\varepsilon_{\widetilde G}:\widetilde G\longrightarrow \{-1,1\}$ defined
in~(\ref{eq:defsignature}) is obtained by composing $s_{\widetilde G}$
with the unique faithful character of $\Z/2\Z$. We set
$G=\widetilde{G}/\langle z\rangle$. 
    Let $\chi\in\Irr(\widetilde G)$. Then either $\chi(z)=\chi(1)$ or
$\chi(z)=-\chi(1)$. We say that $\chi$ is a \emph{non-spin character} in
the first case, and a \emph{spin} character in the second case. We will
denote by $\Irr(\widetilde G)^+$ and $\Irr(\widetilde G)^-$ the sets of
irreducible non-spin characters and spin characters of $\widetilde G$,
respectively. Any $\chi\in\Irr(\widetilde G)^+$ has $z$ in its kernel and
therefore arises as the inflation (or pullback) through the natural
projection $\pi:\widetilde{G}\longrightarrow G$ of an irreducible
character of $G$. Hence, $\Irr(\widetilde G)^+$ can be naturally
identified with $\Irr(G)$.

\subsection{Double covering groups of the Symmetric and Alternating groups}

    Let $n$ be a positive integer. We consider the double covering group
$\tSym_n$ of the symmetric group $\sym_n$ defined by the presentation
$$
\tSym_n=\langle t_1,\ldots,t_{n-1},z\mid z^2=1,\,t_i^2=z,
(t_it_{i+1})^3=z,(t_it_j)^2=z\ (|i-j|\geq 2)
\rangle.
$$
We recall that we have the exact sequence
\begin{equation}
\label{eq:projSn}
1\longrightarrow \langle z\rangle\longrightarrow
\tSym_n\longrightarrow \sym_n\longrightarrow 1,
\end{equation}
where $\pi:\tSym_n\longrightarrow \sym_n$ denotes the natural projection.
For any partition $\lambda$, let $|\lambda|$ denote its size (that is, the
sum of its parts) and let $\ell(\lambda)$ be the number of parts. We then
introduce in the usual way \emph{the sign} of $\lambda$ by
\begin{equation}
\label{eq:sign}
\sgn(\lambda)=(-1)^{|\lambda|-\ell(\lambda)}.
\end{equation}
    We also denote by $\sgn:\Sym_n\longrightarrow \{\pm 1\}$
the usual \emph{signature character} of $\Sym_n$, in other words,
$\sgn(x)$ is equal to the sign of the permutation $x$, which coincides
with the sign of the partition given by the cycle type of $x$.
    We define the \emph{sign character} of $\tSym_n$ by setting
$\varepsilon= \operatorname{sgn}\circ \pi$. Note that
$\varepsilon=\varepsilon_{\tSym_n}$, as defined in
(\ref{eq:defsignature}). In particular, $\tSym_n\in\mathcal G$. We set
$\tAlt_n=\tSym_n^+$. Note that $\sym_n\in\mathcal G$ with respect to
$\operatorname{sgn}$ and that $\sym_n^+=\Alt_n$. Furthermore, we have
$\tAlt_n=\pi^{-1}(\alt_n)$ which is the double covering group of the
alternating group $\alt_n$. 

    We remark that $\Irr(\tSym_n)^+$ and $\Irr(\tSym_n)^-$ are
$\varepsilon$-stable, meaning that the associate characters of a spin,
(resp.\!\!\! non-spin) character are of the same type. It is well known
that the non-spin characters of $\tSym_n$ (that is, the irreducible
characters of $\Sym_n$) are naturally labeled by the set $\mathcal P_n$ of
partitions of $n$. 

    Moreover, in~\cite{schur}, Issai Schur gave a parametrization of the
irreducible spin characters of $\tSym_n$. For this purpose, consider the
set $\mathcal D_n$ of partitions $\lambda$ of $n$ with distinct parts. Any
such partition $\lambda\in\mathcal D_n$ is also called a
\emph{bar-partition} of $n$. The set $\mathcal D_n$ labels the $\langle
\varepsilon\rangle$-orbits in $\Irr(\tSym_n)$. 
    More precisely, let $\mathcal D_n^+$ (resp. $\mathcal D_n^-$) be the
subset of $\lambda\in\mathcal D_n$ consisting of partitions $\lambda$ such
that $\sgn(\lambda)=1$ (resp.\! $-1$). Then, the elements of $\mathcal
D_n^+$ label the $\langle\varepsilon\rangle$-orbits of size $1$, while
those of $\mathcal D_n^-$ label orbits of size $2$.

    To describe the irreducible spin characters of $\tSym_n$ and
$\tAlt_n$, we use the symbol ``$\xi$'' in a manner consistent with the
notation introduced earlier. More precisely, as explained in
\S\ref{subsec:index2}, for $\lambda\in \mathcal D_n$, we denote by 
\begin{equation}
\label{eq:xidef}
\xi_{\lambda},\quad
\xi_{\lambda}^+ \quad\text{and}\quad \xi_{\lambda}^-
\end{equation}
the corresponding characters of either $\tSym_n$ or $\tAlt_n$, depending
on the context.
According to our convention:
\begin{itemize}
\item[$\star$] If $\sgn(\lambda)=1$ then $\xi_{\lambda}$ is a
self-associate character of $\tSym_n$, and $\xi_{\lambda}^+$ and
$\xi_{\lambda}^-$ are two associate spin characters of $\Irr(\tAlt_n)$. 
\item[$\star$] if $\sgn(\lambda)=-1$, then $\xi_{\lambda}^+$ and
$\xi_{\lambda}^-$ are two associate spin characters of $\tSym_n$, and
$\xi_{\lambda}$ is a self-associate character of $\tAlt_n$.
\end{itemize}

\section{Combinatorics of partitions and bar-partitions} 
\label{sec:part2}

\subsection{Frobenius symbol and abacus representation of a partition}
\label{subsec:frob}

    Let $\lambda$ be a partition of a positive integer $n$, and let $t$ be
a positive integer. In this section, we recall how to associate the
Frobenius symbol to $\lambda$.
For our purposes, we follow the approach outlined in~\cite[Section 2]{BrNa2}.

    First, recall that $\lambda$ is completely determined by its
\emph{Young diagram} $[\lambda]$. The boxes of $[\lambda]$ are indexed by
coordinate $(i,j)$ in matrix notation, and a box with coordinate $(j,j)$
is called \emph{a diagonal box} of $\lambda$.  

    To each diagonal box, we associate two nonnegative integers: the
number of boxes to its right in the same row, and the number of boxes
above it in the same column. Collecting these numbers over all diagonal
boxes gives two finite sets of nonnegative integers, denoted
\(\mathcal{A}_\lambda\) and $\mathcal{L}_\lambda$, called the sets of
\emph{arms} and \emph{legs} of $\lambda$, respectively.
The \emph{Frobenius symbol} of \(\lambda\) is then defined as
$$
\mathcal{F}_\lambda = (\mathcal{L}_\lambda \mid \mathcal{A}_\lambda).
$$
    Note that $|\mathcal{L}_\lambda| = |\mathcal{A}_\lambda|$, and
conversely, any pair of finite sets $L$ and $A$ of equal cardinality
uniquely determines a partition $\lambda$ such that
$$
\mathcal{F}_\lambda = (L \mid A).
$$
 An arm and a leg forming a diagonal hook in the Young diagram are said 
 to be \emph{paired}.

\begin{example}
Let $L=\{1,2,3\}$ and $A=\{0,2,3\}$. Then the Young diagram of the
corresponding partition $\lambda$ is
\medskip
\begin{center} 
\frob[(1,0),(2,-1),(3,-2)]
    {{\relax,\relax,\relax,\relax},{\relax,\relax,\relax,\relax},{\relax,\relax,\relax},{\relax,\relax\relax,\relax}}
    {}
    {(2,0),(3,0),(4,0),(3,-1),(4,-1)}
    {(1,-1),(1,-2),(1,-3),(2,-2),(2,-3),(3,-3)}
    {}
    {}
\end{center}
Thus, $\lambda=(4,4,3,3)$. In this example, the paired pairs are 
$$(1,0),\quad (2,2)\quad\text{and}\quad (3,3).$$
\end{example}
\smallskip

    Let $L$ and $A$ be two finite subsets of the non-negative integers. We
associate to them an abacus $\mathcal T_{L,A}$, equipped with a
\emph{fence} $\mathfrak f$. Slots both above and below the fence are
labeled by the set of nonnegative integers. Each slot contains a bead,
which can be either white or black. When $L$ and $A$ have the same
cardinality, the abacus $\mathcal T_{L,A}$ is called \emph{a pointed
abacus}. In this case, the number of black beads above the fence equals
the number of white beads below it.

\begin{notation} 
\label{not:beadcolor} 
    Throughout the following, if $x \in \N$ labels a slot (either above or below
the fence), then let $b_x$ denote the color of the bead occupying that position.
\end{notation}

Given a  partition $\lambda$ with Frobenius symbol $\mathcal
F_{\lambda}=(\mathcal L_{\lambda} \mid \mathcal A_{\lambda})$, we can
uniquely associate to it its \emph{pointed abacus} $\mathcal T_{\mathcal
L_{\lambda},\mathcal A_{\lambda}}$. 

\begin{example}
We give the pointed abacus of $\lambda=(4,4,3,3)$:
\medskip

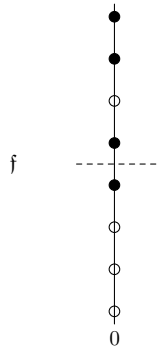
\begin{center}
\begin{tikzpicture}[line cap=round,line join=round,>=triangle
45,x=0.7cm,y=0.7cm, scale=0.8,every node/.style={scale=0.8}]

\draw [dash pattern=on 2pt off 2pt](-1.5,1.5)-- (0.5,1.5);
\draw(-3,1.5)node{$\mathfrak f$};

\draw (-0.9,-2.3) node[anchor=north west] {$0$};

\draw (-0.6,5.3)-- (-0.6,-2.3);

\begin{scriptsize}
	
\draw (-0.6,2)[fill=black] circle (2.5pt);
\draw (-0.6,3) circle (2.5pt);
\draw (-0.6,4)[fill=black] circle (2.5pt);
\draw (-0.6,5)[fill=black] circle (2.5pt);

\draw (-0.6,1)[fill=black] circle (2.5pt);

\draw (-0.6,0)  circle (2.5pt);
\draw (-0.6,-1) circle (2.5pt);
\draw (-0.6,-2) circle (2.5pt);
\end{scriptsize}


\end{tikzpicture}











	







\end{center}
\end{example}

\subsection{Littlewood decomposition of a bar-partition}
\label{subsec:barabaque}

    In this section, we assume that $t$ is \emph{an odd positive integer}.
We begin by introducing the notions of the $t$-bar-abacus, twisted
$t$-bar-abacus, $t$-bar core, $t$-bar quotient and $t$-bar characteristic
vector of a bar-partition. Although the concepts of $t$-bar-abacus and twisted
$t$-bar-abacus are not explicitly presented using this terminology
in~\cite{olsson}, we will essentially follow the approach outlined there.

    Let $\lambda=(\lambda_1,\ldots,\lambda_s)\in\mathcal D_n$. We
associate to $\lambda$ an abacus with $t$ runners labeled from $0$ to
$t-1$. The slots of each runners are labeled by $\N$ from bottom to top.
We begin by placing a white bead to each slot. Then, for each $1\leq j\leq
s$, writing $\lambda_j=tk_j+r_j$ with $0\leq r_j\leq t-1$, we place a
black in the $k_j$-th slot of the $r_j$-th runner. The resulting
$t$-abacus is called the \emph{$t$-bar-abacus} of $\lambda$.

\begin{example}
Consider $\lambda=(1,2,3,5)\in\mathcal D_{10}$. Then its
$3$-bar-abacus is
\medskip
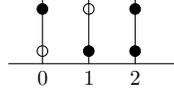
\begin{center}
\begin{tikzpicture}[line cap=round,line join=round,>=triangle
45,x=0.7cm,y=0.7cm, scale=0.8,every node/.style={scale=0.8}]

\draw (-2.5,-0.3)-- (1.5,-0.3);

\draw (-2,-0.3) node[anchor=north west] {$0$};

\draw (-1.7,1.3)-- (-1.7,-0.3);

\begin{scriptsize}

\draw (-1.7,0) circle (2.5pt);

\draw (-1.7,1)[fill=black] circle (2.5pt);

\end{scriptsize}

\draw (-0.9,-0.3) node[anchor=north west] {$1$};

\draw (-0.6,1.3)-- (-0.6,-0.3);

\begin{scriptsize}
	
\draw  (-0.6,0)[fill=black] circle (2.5pt);
\draw (-0.6,1) circle (2.5pt);

\end{scriptsize}

\draw (0.2,-0.3) node[anchor=north west] {$2$};

\draw (0.5,1.3)-- (0.5,-0.3);

\begin{scriptsize}

\draw  (0.5,0)[fill=black] circle (2.5pt);
\draw (0.5,1)[fill=black] circle (2.5pt);

\end{scriptsize}
\end{tikzpicture}
\end{center}
\end{example}
\medskip

\begin{notation}
\label{not:longueurpart} 
    Assume that $x \in \N$ labels a slot on the $r$-th runner
$\mathcal{R}_r$, where $0 \leq r \leq t-1$, of a $t$-bar-abacus, such that
the bead $b_x$ is black. We denote by $d_x$ the length of the part of
$\lambda$ corresponding to $x$, given by
\begin{equation}
\label{eq:longueurpat}
d_x = t x + r.
\end{equation}
\end{notation}

    Now, starting from a $t$-bar-abacus $\mathcal{R}$ with runners
${\mathcal R}_0, \ldots, {\mathcal R}_{t-1}$, we construct a
$\frac{1}{2}(t+1)$-abacus $\mathcal{R}'$, called \emph{the twisted
$t$-bar-abacus of $\lambda$}, defined as follows. Its first runner ${\mathcal
R}'_0$ is simply ${\mathcal R}_0$, and for each $1 \leq r \leq
\frac{1}{2}(t-1)$, the runner ${\mathcal R}'_r$ is equipped with a fence
$\mathfrak{f}$ such that:
\begin{itemize}
\item[$\star$] Above the fence, we place the runner $\mathcal{R}_r$.
\item[$\star$] Let $\mathcal{R}_r^*$ be the runner obtained from
$\mathcal{R}_{t-r}$ by reversing the colors of all its beads. Then, each
bead in slot $j \in \N$ of $\mathcal{R}_r^*$ is placed in the $j$-th slot
below the fence in ${\mathcal R}'_r$.
\end{itemize}

\begin{example}
\label{ex:bar}
Continuing the previous example, the twisted $3$-bar-abacus of
$\lambda=(1,2,3,5)$ is given by
\medskip
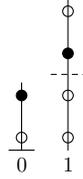
\begin{center}
\begin{tikzpicture}[line cap=round,line join=round,>=triangle
45,x=0.7cm,y=0.7cm, scale=0.8,every node/.style={scale=0.8}]

\draw (-2,-0.3)-- (-1.4,-0.3);
\draw [dash pattern=on 2pt off 2pt](-1,1.5)-- (-0.1,1.5);

\draw (-2,-0.3) node[anchor=north west] {$0$};

\draw (-1.7,1.3)-- (-1.7,-0.3);

\begin{scriptsize}
\draw (-1.7,0) circle (2.5pt);

\draw (-1.7,1)[fill=black] circle (2.5pt);

\end{scriptsize}

\draw (-0.9,-0.3) node[anchor=north west] {$1$};

\draw (-0.6,3.3)-- (-0.6,-0.3);

\begin{scriptsize}

\draw  (-0.6,3)circle (2.5pt);
\draw (-0.6,2)[fill=black]  circle (2.5pt);
	
\draw  (-0.6,0)circle (2.5pt);
\draw (-0.6,1) circle (2.5pt);

\end{scriptsize}

\end{tikzpicture}
\end{center}
\end{example}
\medskip

    Note that the process is bijective: one can recover a bar-partition
$\lambda$ from a twisted $t$-bar-abacus by constructing its
$t$-bar-abacus, where the black beads correspond to the parts of
$\lambda$. 

\begin{example}
Let 
\medskip

\begin{center}
\begin{tabular}{lll}
\begin{tikzpicture}[line cap=round,line join=round,>=triangle
45,x=0.7cm,y=0.7cm, scale=0.8,every node/.style={scale=0.8}]

\draw (-2,-0.3)-- (-1.4,-0.3);
\draw [dash pattern=on 2pt off 2pt](-1,1.5)-- (1,1.5);

\draw (-2,-0.3) node[anchor=north west] {$0$};

\draw (-1.7,1.3)-- (-1.7,-0.3);

\begin{scriptsize}

\draw (-1.7,0) circle (2.5pt);

\draw (-1.7,1)[fill=black] circle (2.5pt);

\end{scriptsize}

\draw (-0.9,-0.3) node[anchor=north west] {$1$};

\draw (-0.6,3.3)-- (-0.6,-0.3);

\begin{scriptsize}
	
\draw  (-0.6,3)[fill=black] circle (2.5pt);
\draw (-0.6,2) circle (2.5pt);
\draw  (-0.6,0)[fill=black] circle (2.5pt);
\draw (-0.6,1) circle (2.5pt);

\end{scriptsize}

\draw (0.2,-0.3) node[anchor=north west] {$2$};

\draw (0.5,3.3)-- (0.5,-0.3);

\begin{scriptsize}
	
\draw  (0.5,3) circle (2.5pt);
\draw (0.5,2) circle (2.5pt);

\draw  (0.5,0) circle (2.5pt);
\draw (0.5,1)[fill=black] circle (2.5pt);

\end{scriptsize}
\end{tikzpicture}
&
\hspace{2cm}
&
\begin{tikzpicture}[line cap=round,line join=round,>=triangle
45,x=0.7cm,y=0.7cm, scale=0.8,every node/.style={scale=0.8}]

\draw (-2,-0.3)-- (3,-0.3);

\draw (-2,-0.3) node[anchor=north west] {$0$};

\draw (-1.7,1.3)-- (-1.7,-0.3);

\begin{scriptsize}

\draw (-1.7,0) circle (2.5pt);

\draw (-1.7,1)[fill=black] circle (2.5pt);

\end{scriptsize}

\draw (-0.9,-0.3) node[anchor=north west] {$1$};

\draw (-0.6,1.3)-- (-0.6,-0.3);

\begin{scriptsize}
	
\draw  (-0.6,0) circle (2.5pt);
\draw (-0.6,1)[fill=black] circle (2.5pt);

\end{scriptsize}

\draw (0.2,-0.3) node[anchor=north west] {$2$};

\draw (0.5,1.3)-- (0.5,-0.3);

\begin{scriptsize}
	
\draw  (0.5,0) circle (2.5pt);
\draw (0.5,1) circle (2.5pt);
\end{scriptsize}

\draw (1.3,-0.3) node[anchor=north west] {$3$};

\draw (1.6,1.3)-- (1.6,-0.3);

\begin{scriptsize}
	
\draw  (1.6,0) circle (2.5pt);
\draw (1.6,1)[fill=black] circle (2.5pt);
\end{scriptsize}

\draw (2.4,-0.3) node[anchor=north west] {$4$};

\draw (2.7,1.3)-- (2.7,-0.3);

\begin{scriptsize}
	
\draw  (2.7,0)[fill=black] circle (2.5pt);
\draw (2.7,1) circle (2.5pt);
\end{scriptsize}
\end{tikzpicture}\\
\hspace{-0.9cm}twisted $5$-bar-abacus&&
\quad$5$-bar-abacus
\end{tabular}
\end{center}
\medskip
On the left side of the picture, a twisted $5$-bar-abacus is shown. On the
right, we display its corresponding $5$-bar abacus. The resulting
bar-partition is $\lambda=(4,5,6,8)$.
\end{example}
\medskip

\begin{notation}
\label{not:abaque}
    Let $m\in\Z$. We associate to $m$ an abacus, denoted by $\mathcal
R(m)$ as follows. We begin with a pointed abacus in which all beads above
the fence $\mathfrak f$ are white, and all beads below $\mathfrak f$ are
black.
If $m>0$, we place a black bead in each of the first $m$ slots above
$\mathfrak f$ (thas is, the slots labeled from $0$ to $m-1$). 
If $m<0$, we place a white bead in each of the first $|m|$ slots  below
$\mathfrak f$ (that is, the slots labeled 
from $0$ to $|m|-1$) below $\mathfrak f$.
In general, if $m\neq 0$, then $\mathcal R(m)$ is not a pointed
abacus.
\end{notation}

    Let $\overline{c} = (c_1, \ldots, c_{(t-1)/2}) \in
\mathbb{Z}^{(t-1)/2}$.  We consider a runner $\mathcal{C}_0'$ whose slots,
labeled by $\mathbb{N}$, are filled exclusively with white beads. For each
$1 \leq r \leq \frac{1}{2}(t-1)$, following Notation~\ref{not:abaque}, we
set $\mathcal{C}_r' = \mathcal{R}(c_r)$.  This yields a twisted
$t$-bar-abacus whose associated bar-partition $\lambda_{\overline{c}}$ is
called a \emph{$t$-bar core partition} with \emph{$t$-bar characteristic
vector} $\overline{c}$.  
\smallskip

    Let $\lambda$ be a bar-partition with twisted $t$-bar-abacus
$\mathcal{R}_0',\,\ldots,\,\mathcal{R}_{(t-1)/2}'$. Since $\mathcal{R}_0'$
is a $1$-bar-abacus, we denote its corresponding bar-partition by
$\lambda^0$. 
    Now, for $1 \leq r \leq \frac 1 2 (t-1)$, although $\mathcal{R}_r'$ is
equipped with a fence, it is in general not a pointed $1$-abacus. However,
there exists a unique way to transform it into one by applying a push
down or a pull up to the runner. As in~\S\ref{subsec:frob}, if we push
down, we let $c_r$ be the number of downward moves. If
we pull up, we set $c_r$ to be the negative of the number of upward
moves. After this transformation, the runner $\mathcal{R}_r'$ becomes a
pointed $1$-abacus, denoted by $\mathcal{S}_r'$, whose corresponding
partition we denote by~$\lambda^r$.

Then we define \emph{the $t$-bar-quotient of $\lambda$} by
$$\lambda^{(\overline
t)}=(\lambda^0,\lambda^1,\ldots,\lambda^{(t-1)/2}),$$
and we write
$$\overline{c}_{t}(\lambda)=(c_1,\ldots, c_{
(t-1)/2})$$
for \emph{the $t$-bar characteristic vector} associated with $\lambda$.
The $t$-bar core partition associated to $\overline c_t(\lambda)$ is
denoted by $\lambda_{(\overline t)}$ and is called the \emph{$t$-bar core
of $\lambda$}.

We also define the \emph{bar $t$-weight} of $\lambda$ by
\begin{equation}
\label{eq:pweightcar}
w_{\overline t}(\lambda)=\sum_{j=0}^{(t-1)/2}|\lambda^j|
\end{equation}
\emph{the bar $t$-weight} of $\lambda$.
The correspondence $\lambda\mapsto (\lambda_{(\overline
t)},\,\lambda^{\,(\overline t\,)})$ is referred to as the \emph{
$t$-bar Littlewood map}.

    We can then define \emph{the $t$-bar cocore partition}
$\lambda^{[\,\overline t\,]}$ associated with $\lambda$, as the
bar-partition of size $tw$ with empty $t$-bar core, and the same $t$-bar
quotient as $\lambda$. 
Then, the tuple 
$$(\lambda_{(\overline p)},\lambda^{[\,\overline p\,]})$$
is called the \emph{$p$-bar Littlewood decomposition} of $\lambda$.

    In the following, a bar-partition $\lambda$ is said to be a
\emph{$t$-bar cocore} if its $t$-bar core is empty.
\begin{remark}
\label{rk:littelabaque}
    By construction, $\mathcal S'=(\mathcal S_0',\ldots,\mathcal
S_{(t-1)/2}')$ is the twisted $t$-bar-abacus of $\lambda^{[\overline t]}$.
We have $\mathcal S_0'=\mathcal R_0'$, and for $1\leq r\leq (t-1)/2$, a
runner $\mathcal R_r'$ is obtained from $\mathcal S_r'$ by applying the
push, or pull, of distance $|c_r|$, where $(c_1,\ldots,c_{(t-1)/2})$ is
the $t$-bar characteristic vector of $\lambda$.
\end{remark}

\begin{example}
Consider the bar-partition $\lambda$ with $3$-bar core $(1)$ and $3$-bar
quotient $((1),(2))$.
We can reconstruct $\lambda$ as follows. First, we display the twisted
$3$-bar abacus of $\lambda_{(\overline 3)}$.
\medskip
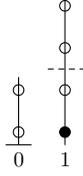
\begin{center}
\begin{tikzpicture}[line cap=round,line join=round,>=triangle
45,x=0.7cm,y=0.7cm, scale=0.8,every node/.style={scale=0.8}]

\draw (-2,-0.3)-- (-1.4,-0.3);
\draw [dash pattern=on 2pt off 2pt](-1,1.5)-- (-0.1,1.5);

\draw (-2,-0.3) node[anchor=north west] {$0$};

\draw (-1.7,1.3)-- (-1.7,-0.3);

\begin{scriptsize}
\draw (-1.7,0) circle (2.5pt);

\draw (-1.7,1) circle (2.5pt);

\end{scriptsize}

\draw (-0.9,-0.3) node[anchor=north west] {$1$};

\draw (-0.6,3.3)-- (-0.6,-0.3);

\begin{scriptsize}

\draw  (-0.6,3)circle (2.5pt);
\draw (-0.6,2)  circle (2.5pt);
	
\draw  (-0.6,0)[fill=black]circle (2.5pt);
\draw (-0.6,1) circle (2.5pt);

\end{scriptsize}
\end{tikzpicture}
\end{center}
We observe that its $3$-bar characteristic vector is $(-1)$. Then, using
the $3$-bar quotient of $\lambda$, we construct the twisted $3$-bar-abacus
$S$ associated to $\lambda^{[\overline 3]}$.
\medskip
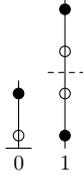
\begin{center}
\begin{tikzpicture}[line cap=round,line join=round,>=triangle
45,x=0.7cm,y=0.7cm, scale=0.8,every node/.style={scale=0.8}]

\draw (-2,-0.3)-- (-1.4,-0.3);
\draw [dash pattern=on 2pt off 2pt](-1,1.5)-- (-0.1,1.5);

\draw (-2,-0.3) node[anchor=north west] {$0$};

\draw (-1.7,1.3)-- (-1.7,-0.3);

\begin{scriptsize}
\draw (-1.7,0) circle (2.5pt);

\draw (-1.7,1)[fill=black] circle (2.5pt);

\end{scriptsize}

\draw (-0.9,-0.3) node[anchor=north west] {$1$};

\draw (-0.6,3.3)-- (-0.6,-0.3);

\begin{scriptsize}

\draw  (-0.6,3)[fill=black]circle (2.5pt);
\draw (-0.6,2)  circle (2.5pt);
	
\draw  (-0.6,0)[fill=black]circle (2.5pt);
\draw (-0.6,1) circle (2.5pt);

\end{scriptsize}

\end{tikzpicture}
\end{center}
    We deduce that $\mathcal R'_0=\mathcal S'_0$ and the runner $\mathcal
R'_1$ is obtained from $\mathcal S_1'$ by a downward push of $1$. We
recognize the twisted $3$-bar-abacus of Example~\ref{ex:bar}, and conclude
that $\lambda=(1,2,4,5)$.
\end{example}

\subsection{Length, sign and $t$-bar-Littlewood decomposition}

    For a bar-partition $\lambda$, we denote by $\ell(\lambda)$ its number
of parts. In this section, we illustrate the tools developed above by
establishing a connection between the length of a bar-partition and those
of its $t$-bar core and $t$-bar cocore. This relationship will be useful
in the sequel. We also emphasize that the technique used in the proof,
particularly the \emph{push-and-pull process}, will play a crucial role
in what follows.

\begin{lemma}
\label{lemma:longueurlittelwoodec} 
Let $t$ be a positive odd integer. For any bar-partition $\lambda$, we
have 
$$\ell(\lambda)-\ell(\lambda_{(\overline
t)})-\ell(\lambda^{[\overline t]}) \equiv 0\mod 2.$$
\end{lemma}

\begin{proof}
    We write $\mathcal R=(\mathcal R_0,\ldots,\mathcal R_{t-1})$,
$\mathcal S=(\mathcal S_0,\ldots,\mathcal S_{t-1})$ and $\mathcal
C=(\mathcal C_0,\ldots,\mathcal C_{t-1})$ for the $t$-bar abacus of
$\lambda$, $\lambda^{[\overline t]}$ and $\lambda_{(\overline t)}$,
respectively. Additionally, we write $\mathcal R'=(\mathcal
R_0',\ldots,\mathcal R_{(t-1)/2}')$, $\mathcal S'=(\mathcal
S_0',\ldots,\mathcal S_{(t-1)/2}')$ and $\mathcal C'=(\mathcal
C_0',\ldots,\mathcal C_{(t-1)/2}')$ for their twisted $t$-bar-abacus.
To prove the result, we will compare the black beads of these $t$-bar
abacus runner by runner. 
We  denote by $\overline{c}=(c_1,\ldots,c_{(t-1)/2})$ the $t$-bar
characteristic vector of $\lambda_{(\overline t)}$. 

    Let $1\leq r\leq (t-1)/2$. We set $\alpha_r$, $\beta_r$ and $\gamma_r$
the numbers of parts with residues $r$ and $t-r$ of $\lambda$,
$\lambda^{[\overline t]}$ and $\lambda_{(\overline t)}$, respectively.

Assume that $c_r>0$. We define
$$D_r=\{x\in\mathcal S_{t-r}\mid 0\leq x\leq c_r-1,\ b_x\ \text{is
black}\},$$
$$
A_r=\{x\in\mathcal S_{t-r}\mid 0\leq x\leq c_r-1,\ b_x\ \text{is
white}\},$$
$$L_r=\{x\in\mathcal S_{t-r}\mid  x\geq c_r,\ b_x\ \text{is
black}\}\text{ and }
M_r=\{x\in\mathcal S_{r}\mid b_x\ \text{is black}\},$$
where $b_x$ is defined in Notation~\ref{not:beadcolor}.
    By construction, the parts of $\lambda^{[\overline t]}$ with residues
$r$ and $t-r$ are labeled by $D_r\cup L_r\cup M_r$. Furthermore, as
explained in Remark~\ref{rk:littelabaque}, $\mathcal R'_r$ is obtained
from $\mathcal S_r'$ by pulling up by $c_r$. In particular, the parts of
$\lambda$ with residues $r$ and $t-r$ are labeled by $A'_r\cup L'_r\cup
M'_r$, where
$$A'_r=\{c_r-1-x \mid x\in A_r\},\ L'_r=\{x-c_r\mid x\in
L_r\}\text{ and } M'_r=\{x+c_r\mid x\in M_r\}.$$

    We define $D_r'=\{c_r-1-x\mid x\in D_r'\}$. Since $c_r>0$,
$\lambda_{(\overline t)}$ has no parts with residues $t-r$, and its parts
with residues $r$ are labeled by $\{0,\ldots,c_r-1\}=A'_r\cup D'_r$.
It follows that
\begin{align}
\label{eq:relationparts}
\alpha_r&=|A'_r\cup L'_r\cup
M'_r|=|A'_r|+|L'_r|+|M'_r|=\gamma_r-|D'_r|+|L_r|+|M_r|\\
&=\gamma_r+\beta_r-|D'_r|-|D_r|\nonumber\\
&=\gamma_r+\beta_r-2|D_r|.\nonumber
\end{align}

    Assume that $c_r<0$. By a similar argument, we can prove that
(\ref{eq:relationparts}) holds, 
where now $D_r=\{x\in \mathcal S_r\mid 0\leq x\leq |c_r|-1,\ b_x\text{ is
black}\}$. Additionally, we observe that $\lambda_{(\overline t)}$ has no
parts with residue $0$, and that $\mathcal R_0=\mathcal S_0$, which means
that $\lambda$ and $\lambda^{[\overline t]}$ share the same parts with
residue $0$. Similarly, if $c_r=0$ for some $1\leq r\leq (t-1)/2$, then
$\lambda_{(\overline t)}$ has no parts with residue $r$ or $t-r$, and
$\mathcal R'_r=\mathcal S'_r$, implying that $\lambda$ and
$\lambda^{[\overline t]}$ have the same parts with residue $r$ and $t-r$. 
Finally, by summing the equalities~(\ref{eq:relationparts}) 
we obtain that
\begin{equation}
\label{eq:longeurrelation}
\ell(\lambda)=\ell(\lambda_{(\overline t)})+\ell(\lambda^{[\overline
t]})-2d,
\end{equation}
where $d=\sum_{r=1}^{(t-1)/2}|D_r|$. Hence, the result follows.
\end{proof}

\begin{corollary}
\label{cor:signlitteldec}
Let $\lambda$ be a bar-partition and $t$ be a positive odd integer. Then
$$\operatorname{sgn}(\lambda)=\operatorname{sgn}(\lambda_{(\overline t)})
\operatorname{sgn}(\lambda^{[\overline t]}).$$
\end{corollary}

\begin{proof}
By~\cite[Corollary 4.4]{olsson}, we have 
\begin{equation}
\label{eq:sizerelation}
|\lambda|=|\lambda_{(\overline t)}|+|\lambda^{[\overline t]}|.
\end{equation}
Now, using Lemma~\ref{lemma:longueurlittelwoodec}, we obtain
\begin{align*}
\operatorname{sgn}(\lambda)&=(-1)^{|\lambda|-\ell(\lambda)}=
(-1)^{|\lambda_{(t)}|+|\lambda^{[\overline
t]}|-\ell(\lambda_{(\overline t)}) -\ell(\lambda^{[\overline t]})+2d}\\
&=(-1)^{|\lambda_{(t)}|-\ell(\lambda_{(\overline t)})}(-1)^{|\lambda^{[\overline
t]}|-\ell(\lambda^{[\overline t]})}\\
&=\operatorname{sgn}(\lambda_{(\overline
t)})\operatorname{sgn}(\lambda^{[\overline t]}),
\end{align*}
as required.
\end{proof}

\section{Action of Galois-Navarro automorphisms}
\label{sec:actiongalois}

\subsection{Results on the Jacobi symbol}
\label{subsec:actionspin}

    Let $f$ be any Galois automorphism of $\overline \Q$ over $\Q$. When
$\alpha$ is an algebraic number of degree $2$, we denote by
$\tau(\alpha,f)\in \{-1,1\}$ the value such that
$f(\alpha)=\tau(\alpha,f)\alpha$.

    Let $m$ be an integer and $\omega$ be a primitive $m$-root of unity.
In particular, there exists a positive integer $r$ such that
$f(\omega)=\omega^r$. Proposition 4.1 of~\cite{BrNa} asserts that, if $m$
is odd, then
\begin{equation}
\label{eq:actgalroot}
\tau(\sqrt m,f)=\tau(i,f)^{(m-1)/2}\legendre{r}{m},
\end{equation}
where $\legendre{\cdot}{\cdot}$ denotes the Jacobi symbol. \medskip

    Now, let us focus on the case of the Galois automorphism $\sigma_p$.
To simplify the notation, for all positive integer $m$ we will write 
\begin{equation}
\label{eq:epsigmap}
\tau(m)=\tau(\sqrt m,\sigma_p).
\end{equation}

\begin{lemma}
If $m$ is a positive integer, then
$$\tau(m)=\legendre{m}{p}.$$
\end{lemma}

\begin{proof}
    Write $m=2^km'$ with $m'$ odd. First, we observe that
$\sqrt{2}=e^{i\frac{\pi}{4}}+e^{-i\frac{\pi}4}$ is a sum of two $p'$-roots
of unity. Hence, $\sigma_p(\sqrt
2)=e^{i\frac{\pi}{4}p}+e^{-i\frac{\pi}4p}$. Next, by studying the residues
of $p$ modulo $8$, and using~\cite[Proposition 5.1.3]{KennethRosenNrTh},
we deduce that
\begin{equation}
\label{eq:sqrt2}
\tau(2)=(-1)^{\frac{p^2-1}8}=\legendre{2}{p}.
\end{equation}
    On the other hand, $i$ is a $p'$-root of unity, and we have
$\sigma_p(i)=i^p$ which equals $i$ if $(p-1)/2$ is even and $-i$
otherwise. In particular, we obtain
\begin{equation}
\label{eq:ibouge}
\tau(i)=(-1)^{(p-1)/2}=\legendre{-1}{p}.
\end{equation}
    Now, by using~(\ref{eq:actgalroot}) and the quadratic reciprocity
formula, we find
\begin{equation}
\label{eq:m'}
\tau(m')=(-1)^{(p-1)(m'-1)/4}\legendre{p}{m'}=\legendre{m'}{p}.
\end{equation}
    Finally, using (\ref{eq:sqrt2}), (\ref{eq:m'}) and the fact that
$\tau$ is multiplicative, we deduce that
$$\tau(m)=\tau(2)^k\tau(m')=\legendre{2}{p}^k\legendre{m'}{p}
=\legendre{2^km'}{p}=\legendre{m}{p},$$
as required.
\end{proof}

\subsection{Character values}
\label{subsec:spinvalues}
\smallskip

    Let $n$ be a positive integer. Let $\lambda\in\mathcal P_n$. We choose
an element $s_{\lambda}\in\Sym_n$ with cycle type $\lambda$. We fix
$t_{\lambda}\in\tSym_n$ such that $\pi(t_{\lambda})=s_{\lambda}$, where
$\pi:\tSym_n\longrightarrow \Sym_n$ is natural projection given in
(\ref{eq:projSn}). Such an element is said to have \emph{cycle type}
$\lambda$. Note that the two elements $t_{\lambda}$ and $zt_{\lambda}$
have the same cycle type. Furthermore, if $t_{\lambda}$ and $zt_{\lambda}$
are conjugate (either in $\tSym_n$ or in $\tAlt_n$), then the conjugacy
class of $t_{\lambda}$ in $\tSym_n$ or in $\tAlt_n$ is called a
\emph{non-split class}. Otherwise, it is referred to as a \emph{split
class}. 

    Let $\mathcal{O}_n$ denote the set of $\lambda\in\mathcal{P}_n$ with
parts of odd length. Then Schur proved (see~\cite{schur}, \S7 and p.\,176)
that the $\tSym_n$-split classes are labeled by $\mathcal{O}_n\cup
\mathcal{D}^-_n$, and the $\tAlt_n$-split classes by $\mathcal \mathcal
O_n \cup \mathcal D_n^+$. 

    Now, we recall some facts about the values of spin characters and
difference characters of $\tSym_n$ and $\tAlt_n$. We will use the notation
introduced in equations (\ref{eq:xidef}) and (\ref{eq:chardiff}).

    First, any spin characters vanish on non-split classes. Then Schur
proved in \cite{schur} the following result. For any
$\lambda=(\lambda_1,\ldots,\lambda_k)\in \mathcal D_n$, the spin character
$\xi_{\lambda}$ in $\Irr(\tSym_n)$ or $\Irr(\tAlt_n)$ takes only integer
values.
    The difference character corresponding to $\lambda$ vanishes on all
conjugacy classes, expect those whose representatives have cycle type
$\lambda$. More precisely, in the case of $\tSym_n$, there are two such
conjugacy classes, and the non-zero values of the difference character are
equal to
\begin{equation}
\label{eq:diffSntilte}
a_{\lambda}=
\sqrt 2\,i^{\frac{n-k+1}2}\sqrt{\lambda_1\cdots\lambda_k},
\end{equation}
up to a sign.
    Similarly,  in the case of $\tAlt_n$, there are two or four classes,
and again, the non-zero values of the difference character are, up to
a sign, equal to
\begin{equation}
\label{eq:diffAntilte}
b_{\lambda}=i^{\frac{n-k}2}\sqrt{\lambda_1\cdots\lambda_k},
\end{equation}

\subsection{Action of Galois automorphisms on spin characters}

    Let $p$ be a odd prime number. Recall that the group of Navarro
automorphisms $\mathcal H$ is generated by the Galois automorphisms $f$
that act trivially on the set of $p'$-roots of unity, together with the
Galois automorphism $\sigma_p$, that acts by $x\mapsto x^p$ on $p'$-roots
and fixes any complex roots of unity with order a power of $p$. 

    Let $\lambda\in\mathcal D_n$, and $f\in\mathcal H$. If
$\lambda\in\mathcal D_n^+$, then $\xi_{\lambda}\in\Irr(\tSym_n)$ is fixed
by $f$, since $\xi_{\lambda}$ takes only integer values. It follows that
$f$ permutes the pair $\{\xi_\lambda^+,\,\xi_{\lambda}^-\}\subseteq
\Irr(\tAlt_n)$. 
    Similarly, if $\lambda\in\mathcal D_n^-$, then
$\xi_{\lambda}\in\Irr(\tAlt_n)$ is also integer-valued and hence fixed by
$f$. Therefore, \cite[Lemma 2.2]{BrNa} implies that $f$ acts on the pair
$\{\xi_\lambda^+,\,\xi_{\lambda}^-\}\subseteq
\Irr(\tSym_n)$. 

In both cases, there exists a sign $\tau(\lambda,f)$ such that
\begin{equation}
\label{eq:Sntildemove}
f(\xi_\lambda^{\pm})=\xi_\lambda^{\pm\tau(\lambda,f)}.
\end{equation}
In term of
difference character, this gives
\begin{equation}
\label{eq:Antildemove}
f(\Delta_\lambda)=\tau(\lambda,f)\Delta_{\lambda}.
\end{equation}

    Now, observe that the complex numbers defined in equations
(\ref{eq:diffAntilte}) and (\ref{eq:diffSntilte}) have degree $2$.
Therefore, using the notation introduced in \S\ref{subsec:actionspin}, we
obtain
$$\tau(\lambda,f)=\tau(a_{\lambda},f),$$ when $\lambda\in \mathcal D_n^-$,
and
$$\tau(\lambda,f)=\tau(b_{\lambda},f),$$
when $\lambda\in\mathcal D_n^+$.

\subsection{Bar-abacus pairing}
\label{subsection:pairing}

    In this section, we introduce a key concept, which we call
\emph{bar-abacus pairing}, that will play a fundamental role in our
approach. Throughout, $p$ denotes an odd prime number.

    Let $\lambda$ be a $p$-cocore bar-partition with $p$-bar quotient
$$\lambda^{(\overline
p)}=(\lambda^0,\lambda^1,\ldots,\lambda^{(p-1)/2}).$$
    We denote by $\mathcal S=(\mathcal S_0,\ldots,\mathcal S_{p-1})$ and
$\mathcal S'=(\mathcal S_0',\ldots,\mathcal S_{(p-1)/2}')$ the $p$-bar
abacus and the twisted $p$-bar-abacus of $\lambda$, respectively.
Fix an integer $1\leq r\leq (p-1)/2$, and let $x\in \N$ be such that $b_x$
is black on $\mathcal S_r$ (see Notation~\ref{not:beadcolor}). In
particular, $x$ labels a part of $\lambda$ of length $d_x$ (see
Notation~\ref{not:longueurpart}) with residue $r$ modulo $p$. 
    Moreover, since $\lambda$ is a $p$-cocore bar-partition, the component
$\mathcal S_r'$ forms the pointed abacus of the partition $\lambda^r$,
with $\mathcal S_r$ positioned above the fence. In particular, $x$ also
labels an arm of $\lambda^r$.
    We denote by $x^*$ the position of the leg of $\lambda^r$ paired with
this arm. After applying the process described in \S\ref{subsec:barabaque}
that transforms the twisted $p$-bar abacus into the $p$-bar abacus, this
leg yields a black bead at position $x^*$ on the runner $\mathcal
S_{p-r}$, and hence corresponds to a part of $\lambda$ of length $d_{x^*}$
with residue $p-r-1$ modulo $p$. 
In this situation, we say the parts $d_x$ and $d_{x^*}$ of $\lambda$ are
\emph{paired parts}.
\label{brothers}

\begin{example}
\label{ex:brothers}
Let $\lambda=(2,3,4,7,9,13,15)$. We observe that $\lambda$ is a $5$-cocore
partition with $p$-bar quotient
$$\lambda^{(\overline 5)}=(\emptyset, (2,1,1),\,(3,2)).$$
The partition $(2,1,1)$ has one diagonal hook, with arm $1$, and leg $2$.
The partition $( 2 , 1 , 1)$ has a diagonal hook with leg $2$ and arm
$1$, whereas the partition $( 3 , 2)$ has two diagonal hooks.
The first has leg $1$ and arm $2$, and the second has leg $0$ and arm $0$. 
The pairs of beads corresponding to these diagonal hooks are illustrated
in the following figure.
\medskip

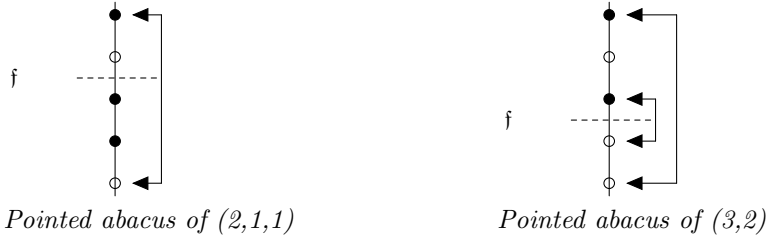
\begin{center}
\begin{tabular}{lll}

\begin{tikzpicture}[line cap=round,line join=round,>=triangle
45,x=0.7cm,y=0.7cm, scale=0.8,every node/.style={scale=0.8}]

\draw [dash pattern=on 2pt off 2pt](-1.5,1.5)-- (0.5,1.5);
\draw(-3,1.5)node{$\mathfrak f$};


\draw (-0.6,3.3)-- (-0.6,-1.3);

\begin{scriptsize}
	
\draw (-0.6,2) circle (2.5pt);
\draw (-0.6,3)[fill=black] circle (2.5pt);

\draw (-0.6,1)[fill=black] circle (2.5pt);

\draw (-0.6,0)[fill=black]  circle (2.5pt);
\draw (-0.6,-1) circle (2.5pt);
\end{scriptsize}

\draw (0.5,3) -- (0.5,-1); 
\draw[->] (0.5,3) -- (-0.2,3);
\draw[->] (0.5,-1) -- (-0.2,-1);
\end{tikzpicture}
&
\hspace{2cm}
&
\begin{tikzpicture}[line cap=round,line join=round,>=triangle
45,x=0.7cm,y=0.7cm, scale=0.8,every node/.style={scale=0.8}]

\draw [dash pattern=on 2pt off 2pt](-1.5,1.5)-- (0.5,1.5);
\draw(-3,1.5)node{$\mathfrak f$};


\draw (-0.6,4.3)-- (-0.6,-0.3);

\begin{scriptsize}
	
\draw (-0.6,2)[fill=black] circle (2.5pt);
\draw (-0.6,3) circle (2.5pt);
\draw (-0.6,4)[fill=black] circle (2.5pt);

\draw (-0.6,1) circle (2.5pt);

\draw (-0.6,0) circle (2.5pt);
\end{scriptsize}

\draw (1,4) -- (1,-0); 
\draw (0.5,2) -- (0.5,1); 
\draw[->] (0.5,2) -- (-0.2,2);
\draw[->] (0.5,1) -- (-0.2,1);
\draw[->] (1,4) -- (-0.2,4);
\draw[->] (1,0) -- (-0.2,0);
\end{tikzpicture}
\\
\centering Pointed abacus of (2,1,1)&&
\centering Pointed abacus of (3,2)
\end{tabular}
\end{center}
\medskip

Consequently, the $5$-bar abacus associated with $\lambda$ is:
\medskip

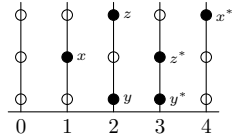
\begin{center}
\begin{tikzpicture}[line cap=round,line join=round,>=triangle
45,x=0.7cm,y=0.7cm, scale=0.8,every node/.style={scale=0.8}]

\draw (-2,-0.3)-- (3,-0.3);

\draw (-2,-0.3) node[anchor=north west] {$0$};

\draw (-1.7,2.3)-- (-1.7,-0.3);

\begin{scriptsize}

\draw (-1.7,0) circle (2.5pt);

\draw (-1.7,1) circle (2.5pt);

\draw (-1.7,2) circle (2.5pt);

\end{scriptsize}

\draw (-0.9,-0.3) node[anchor=north west] {$1$};

\draw (-0.6,2.3)-- (-0.6,-0.3);

\begin{scriptsize}
	
\draw  (-0.6,0) circle (2.5pt);
\draw (-0.6,1)[fill=black] circle (2.5pt);
\node[anchor=west] at (-0.6,1) {\! $x$};
\draw (-0.6,2) circle (2.5pt);

\end{scriptsize}

\draw (0.2,-0.3) node[anchor=north west] {$2$};

\draw (0.5,2.3)-- (0.5,-0.3);

\begin{scriptsize}
	
\draw  (0.5,0)[fill=black] circle (2.5pt);
\node[anchor=west] at (0.5,0) {\! $y$};
\draw (0.5,1) circle (2.5pt);
\draw (0.5,2)[fill=black] circle (2.5pt);
\node[anchor=west] at (0.5,2) {\! $z$};
\end{scriptsize}

\draw (1.3,-0.3) node[anchor=north west] {$3$};

\draw (1.6,2.3)-- (1.6,-0.3);

\begin{scriptsize}
	
\draw  (1.6,0)[fill=black] circle (2.5pt);
\node[anchor=west] at (1.6,0) {\! $y^*$};
\draw (1.6,1)[fill=black] circle (2.5pt);
\node[anchor=west] at (1.6,1) {\! $z^*$};
\draw (1.6,2) circle (2.5pt);
\end{scriptsize}

\draw (2.4,-0.3) node[anchor=north west] {$4$};

\draw (2.7,2.3)-- (2.7,-0.3);

\begin{scriptsize}
	
\draw  (2.7,0) circle (2.5pt);
\draw (2.7,1) circle (2.5pt);
\draw (2.7,2)[fill=black] circle (2.5pt);
\node[anchor=west] at (2.7,2) {\! $x^*$};
\end{scriptsize}
\end{tikzpicture}
\end{center}

Hence, the paired parts of $\lambda$ are
$$(d_x,d_{x^*})=(6,14),\quad
(d_y,d_{y^*})=(2,3)\quad\text{and}\quad
(d_z,d_{z^*})=(12,8).$$
\medskip
\end{example}

\subsection{Galois automorphisms and the bar-Littlewood decomposition}

    Let $n$ be a positive integer, $f\in \mathcal H$, and
$\lambda\in\mathcal D_n$.
In the previous section, we saw that the action of $\mathcal H$ on the
difference character $\Delta_{\lambda}$ is entirely determined by the
action of $f$ on $a_{\lambda}$ or $b_{\lambda}$, depending on whether
$\lambda\in \mathcal D_n^-$ or $\lambda\in \mathcal D_n^+$, respectively.
    For $0\leq r\leq (p-1)$ and a bar-partition $\lambda$, we define
$\mathcal Z_r(\lambda)$ as the set of parts of $\lambda$ with residue
$r$.

    We now present one of the main results of the paper, namely an
explicit computation of $\tau(a_\lambda, f)$ and $ \tau(b_\lambda, f)$
expressed in terms of the Littlewood bar-decomposition of $\lambda $.
Its proof illustrates the effectiveness of the concept of paired pairs
introduced earlier. 

\begin{theorem}
\label{thm:little}
Let $p$ be an odd prime number. 
Let $\lambda$ be a bar-partition.
\begin{enumerate}[(i)]
\item If $\sgn(\lambda_{(\overline p)})=1$ or  $\sgn(\lambda^{[\overline
p]})=1$, then
$$\tau(\lambda,\sigma_p)=\tau(\lambda_{(\overline p)},\sigma_p)
\tau(\lambda^{[\overline p]},\sigma_p).$$
\item If $\sgn(\lambda_{(\overline p)})=-1$ and  $\sgn(\lambda^{[\overline
p]})=-1$, then
$$\tau(\lambda,\sigma_p)=(-1)^{(p-1)/2}\tau(\lambda_{(\overline p)},\sigma_p)
\tau(\lambda^{[\overline p]},\sigma_p).$$
\item If $f$ is any Galois automorphism that acts trivially on the
$p'$-roots of unity, then
$$\tau(\lambda,f)=\tau(\lambda_{(\overline p)},f)
\tau(\lambda^{[\overline p]},f).$$
\end{enumerate}
\end{theorem}

\begin{proof}
    We employ the same strategy as in the proof of
Lemma~\ref{lemma:longueurlittelwoodec}. Let
$\lambda=(\lambda_1,\ldots,\lambda_w)$ be a bar-partition, with $p$-bar
core $\lambda_{(\overline p)}=(\nu_1,\ldots,\nu_u)$ and $p$-cocore
bar-partition $\lambda^{[\overline p]}=(\mu_1,\ldots,\mu_v)$.
Let $\mathcal R=(\mathcal R_0,\ldots,\mathcal R_{p-1})$, $\mathcal
S=(\mathcal S_0,\ldots,\mathcal S_{p-1})$ and $\mathcal C=(\mathcal
C_0,\ldots,\mathcal C_{p-1})$ denote the $p$-bar abacus of $\lambda$,
$\lambda^{[\overline p]}$ and $\lambda_{(\overline p)}$, respectively. 
As before, we write $\mathcal R'=(\mathcal R_0',\ldots,\mathcal
R_{(t-1)/2}')$, $\mathcal S'=(\mathcal S_0',\ldots,\mathcal S_{(t-1)/2}')$
and $\mathcal C'=(\mathcal C_0',\ldots,\mathcal C_{(t-1)/2}')$ for their
twisted $t$-bar-abacus. Let $\overline c=(c_1,\ldots,c_{(p-1)/2})$ be the
$p$-bar characteristic vector of $\lambda_{(\overline p)}$.

Let $1\leq r\leq (p-1)/2$.
Assume that $c_r=0$. Then $C'_r=\emptyset$ and $\mathcal R'_r=\mathcal
S'_r$. Hence, $\lambda_{(\overline p)}$ has no parts of residues $r$ or
$p-r$, and $\lambda$ and $\lambda^{[\overline p]}$ have the same parts of
residues $r$ and $p-r$, which contribute in the same way to the Jacobi
symbol.

    Assume that $c_r>0$. As in Lemma~\ref{lemma:longueurlittelwoodec}, we
now consider the sets $D_r$, $A_r$, $L_r$, $M_r$, $D'_r$, $A'_r$, $L'_r$
and $M'_r$.
Let $x\in M_r$. Using the bar-abacus pairing introduced in
Section~\ref{subsection:pairing}, 
there exists $x^*\in\mathcal S_{p-r}$ such that $b_{x^*}$ is black.
After pulling up by $c_r$, $x$ is transformed into $x+c_r\in M'_r$ that
labels a part of $\lambda$ with residue $r$. Two cases can then arise for
$x^*$. Either: 
\begin{enumerate}[(1)]
\item $x^*\in L_r$, in which case $x^*-c_r\in L'_r$ still labels a part of
$\lambda$ with residue $p-r$.
It follows that
\begin{equation}
\label{eq:preuvepart1}
\legendre{d_{x+c_r}d_{x^*-c_r}}{p}=\legendre{-r^2}{p}=\legendre{d_xd_{x^*}}{p}.
\end{equation}
\item $x^*\in D_r$, and the slot $c_r-1-x^*$ of $\mathcal R_r$ has a white
bead, so its labels no parts of $\lambda$. 
\end{enumerate}
    In the second case, $\lambda$ has lost a part of residue $p-r$
compared to $\lambda^{[\overline p]}$. Note that, since $c_r>0$, there is
a black bead on $C_r$ at the position $c_r-1-x^*\in D'_r$, and the
corresponding part $d_{c_r-1-x^*}$ of $\lambda_{(\overline p)}$ is neither
in $\lambda$ nor in $\lambda^{[\overline p]}$. It follows that
\begin{equation}
\label{eq:preuvepart2}
\legendre{d_{x+c_r}}{p}
\legendre{d_{c_r-1-x^*}}{p}\legendre{d_xd_{x^*}}{p}=\legendre{-1}{p},
\end{equation}
since $d_{x+c_r}d_{c_r-1-x^*}$ is a square modulo $p$.

    On the other hand, if $y\in A_r$, then it labels no parts of
$\lambda^{[\overline p]}$. After the pull, $y$ is transformed into
$c_r-1-y\in A'_r$, and a part with residue $r$ appears in $\lambda$
compared to $\lambda^{[\overline p]}$. Note that there is a black bead on
$C_r$ at the position $c_r-1-y$. Thus both $\lambda$ and
$\lambda_{(\overline p)}$ have the part $d_{c_r-1-y}$. These two parts
then have the same contribution in the Jacobi symbol.

    The case $c_r<0$ is similar, and (\ref{eq:preuvepart1}) and
(\ref{eq:preuvepart2}) still hold, after adapting the definition of
$D'_r$, $A'_r$, $L'_r$ and $M'_r$ using $|c_r|$ instead of $c_r$ as in the
proof of Lemma~\ref{lemma:longueurlittelwoodec}.

    Finally, since $C_0$ is empty and $\mathcal R_0=\mathcal S_0$, the
parts with residue $0$ in both $\lambda$ and $\lambda^{[\overline p]}$ are
the same and contribute similarly to the Jacobi symbol.

In summary, we obtain
\small
\begin{equation}
\label{eq:intpreuve}
\tau\left(\sqrt{\prod_{i=1}^w}\lambda_i,\sigma_p\right)
\tau\left(\sqrt{\prod_{i=1}^u}\nu_i,\sigma_p\right)\tau\left(\sqrt{\prod_{i=1}^v}\mu_i,\sigma_p\right)
=
\prod_{r=1}^{(p-1)/2}\prod_{x\in
D_r}\legendre{-1}{p}=(-1)^{\frac{(p-1)}{2}d},
\end{equation}
\normalsize
where $d=\sum_{r=1}^{(t-1)/2}|D_r|$.\smallskip

    We will now discuss according to the signs of $\lambda$,
$\lambda_{(\overline p)}$ and $\lambda^{[\overline p]}$.
Assume that $\sgn(\lambda)=-1$. Then Lemma~\ref{cor:signlitteldec} implies
that either $\sgn(\lambda_{(\overline p)})=1$ and
$\sgn(\lambda^{[\overline p]})=-1$, or $\sgn(\lambda_{(\overline p)})=-1$
and $\sgn(\lambda^{[\overline p]})=1$.
    In the first case, by (\ref{eq:diffSntilte}), (\ref{eq:diffAntilte}),
(\ref{eq:longeurrelation}), and~(\ref{eq:sizerelation}), the power of
$\tau(i,\sigma_p)$ that appears in the expression
$$\tau(\lambda,\sigma_p)\tau(\lambda_{(\overline
p)},\sigma_p)\tau(\lambda^{[\overline p]},\sigma_p)$$ is given by
$$\frac{1}{2}\left(|\lambda|-l(\lambda)+1-(|\lambda_{(\overline
p)})|-\ell(\lambda_{(\overline p)})\right)-( [\lambda^{[\overline
p]}|-\ell(\lambda^{[\overline p]})+1))=d.$$
    The same relation holds in the second case and if $\sgn(\lambda)=1$
and $\sgn(\lambda_{(\overline p)})=1$ and $\sgn(\lambda^{[\overline
p]})=1$.
When $\sgn(\lambda)=1$, and $\sgn(\lambda_{(\overline p)})=-1$ and
$\sgn(\lambda^{[\overline p]})=-1$, then the power is $d-1$. 

    Note that, when one of the partitions $\lambda$, $\lambda_{(\overline
p)}$ and $\lambda^{[\overline p]}$ has negative sign,
Corollary~\ref{cor:signlitteldec} implies that in fact, exactly two of
these partitions have sign $-1$. Then by (\ref{eq:diffSntilte}) and
(\ref{eq:diffAntilte}), the quantity $\tau(\sqrt 2,\sigma_p)$ appears
twice in $\tau(\lambda,\sigma_p) \tau(\lambda_{(\overline
p)},\sigma_p)\tau(\lambda^{[\overline p]},\sigma_p)$, so it does not
affect the result of the product.

Now, using (\ref{eq:ibouge}) and (\ref{eq:intpreuve}), we deduce that
$$
\tau(\lambda,\sigma_p)\tau(\lambda_{(\overline p)},\sigma_p)
\tau(\lambda^{[\overline p]},\sigma_p)
=\tau(i,\sigma_p)^{d}(-1)^{\frac{(p-1)}{2}d}=(-1)^{\frac{(p-1)}{2}d}(-1)^{\frac{(p-1)}{2}d}=1$$
in the first three cases, and $(-1)^{(p-1)/2}$ in the last case, as
required.

    Finally, (iii) presents no difficulty. As $f$ acts trivially on the
$p'$-roots of unity, it only acts on the parts with residue $0$, which
appear on the first runner of the $p$-bar abacus. However, $\mathcal
C_0=\emptyset$, and $\mathcal R_0=\mathcal S_0$.
The result follows.
\end{proof}

\section{Humphreys product of groups}
\label{sec:humphreys}

    In this section, we consider two finite groups $\widetilde{G}_1$ and
$\widetilde{G}_2$ belonging to $\mathcal G$. For simplicity of notation,
we use the same symbol $s$ the denote the group homomorphisms
$s_{\widetilde{G}_1}$ and $s_{\widetilde G_2}$ defined in
(\ref{eq:stilde}). The central elements of $\widetilde{G}_1$ and
$\widetilde{G}_2$ which arise from the fact that these groups belong to
$\mathcal G$, are also denoted by the same symbol~$z$.

    We recall that in~\cite{Humphreys}, Humphreys endowed the Cartesian
product $\widetilde{G}_1\times\widetilde{G}_2$ with a group structure, the
multiplication of which is given by 
\begin{equation}
\label{eq:loigamma}
(g_1,g_2)(g'_1,g_2')=(z^{s(g_2)s(g'_1)}g_1 g'_1,g_2g'_2),
\end{equation}
and showed that its quotient by $Z=\{(1,1),(z,z)\}$, denoted by
$\widetilde{G}_1\widehat{\times}\widetilde{G}_2$, belongs in~$\mathcal G$.

    Furthermore, he also described the irreducible characters of
$\widetilde{G}_1\widehat{\times}\widetilde{G}_2$. For any characters
$\chi_1\in\Irr(\widetilde{G}_1)$ and $\chi_2\in\Irr(\widetilde G_2)$, he
constructed a character $\chi_1\widehat{\otimes}\chi_2$ of
$\Irr(\widetilde{G}_1\widehat{\times}\widetilde{G}_2)$ such that
\begin{itemize}
\item[$\star$] Each spin character in
$\Irr(\widetilde{G}_1\widehat{\times}\widetilde{G}_2)^-$ is of the form
$\chi_1\widehat{\otimes}\chi_2$ for $\chi_1\in\Irr(\widetilde{G}_1)^-$ and
$\chi_2\in\Irr(\widetilde{G}_2)^-$.
\item[$\star$] Each non-spin character of
$\Irr(\widetilde{G}_1\widehat{\times}\widetilde{G}_2)^+$ is of the form
$\chi_1\widehat{\otimes}\chi_2$ for $\chi_1\in\Irr(\widetilde{G}_1)^+$ and
$\chi_2\in\Irr(\widetilde{G}_2)^+$. In this case,
$\chi_1\widehat{\otimes}\chi_2=\chi_1\otimes\chi_2$.
\end{itemize}
\smallskip

\subsection{Irreducible spin characters of the twisted Humphreys product.}
\label{subsec:irrhumphreys}

    We write $G=\widetilde G_1\widehat{\times}\widetilde{G}_2$. Let $V$
and $W$ be irreducible spin representations of $\widetilde G_1$ and
$\widetilde G_2$, with corresponding characters $\chi_V$ and
$\chi_W$.\medskip

    First, we assume that $V$ is self-associate and that $W$ is
non-self-associate. When we restrict $V$ to the subgroup
$\widetilde{G}_1^+$, it splits into the direct sum of the two irreducible
spin representations, $V^+$ and $V^-$. We define a function
$t:V_1\longrightarrow \{0,1\}$ by setting $t(v)=0$ if $v\in V^+$ and
$t(v)=1$ if $v\in V^-$. Then as described in~\cite[p.452]{Humphreys}, for
each element $(x,y)\in G$ and $(v,w)\in V\times W$, the following formula
defines an action:
\begin{equation}
\label{eq:repself}
(x,y)(v\widehat\otimes
w)=(-1)^{t(v)s(y)}xv\,\widehat\otimes\, y w
\end{equation}
    When extended linearly, this gives an irreducible non-self-associate
spin representation of $G$, whose character is denoted by
$\chi_{V,W}$. 

    Let $(v_1,\ldots,v_m)$ be a basis of $V^+$. Let
$c\in\widetilde{G}_1\,\backslash\,\widetilde G_1^+$.
According to Clifford theory, we have $V^-=cV^+$, and therefore
$(cv_1,\ldots,cv_m)$ forms a basis of $V^-$.
    Hence, for any $h\in \widetilde{G}_1^+$ and $v\in V^+$, the action of
$h$ acts on $V^-$ is given by
$$hc v=c({}^{c^{-1}}h)v,$$
where ${}^{c^{-1}}h=c^{-1}hc$ denotes the conjugation of $h$ by $c^{-1}$.
In particular, if $(v_1,\ldots,v_m)$ is a basis of $V^+$, and $R(h)$
denotes the matrix of $h$ with respect to this basis, then the matrix
representing the action of $h$ on $V^-$ with respect to the basis $(c
v_1,\ldots,c v_m)$ is $R({}^{c^{-1}}h)$.

    Let $(w_1,\ldots,w_r)$ be a basis of $W$. Then $(v_i\otimes
w_j,\,cv_i\otimes w_j)$ for $1\leq i\leq m$ and $1\leq j\leq r$ forms a
basis of $V\otimes W$. For $(x,y)\in G$, we denote by $R(x,y)$ the matrix
representing the action of $(x,y)$ on $V\widehat\otimes W$ with respect to
this basis. For $g\in\widetilde{G}_2$, we denote by $R'(g)$ the matrix
representing the action of $g$ on $W$ with respect to the basis
$(w_1,\ldots,w_r)$.

    Let $(x,y)\in \widetilde{G}_1^+\times\widetilde{G}_2^+$, and
$c'\in\widetilde{G}_2\,\backslash\,\widetilde{G}_2^+$. Then we have 
$$
R(x,y)=
\begin{pmatrix}
R(x)\otimes R'(y)&0\\
0&R({}^{c^{-1}}x)\otimes R'(y)
\end{pmatrix}$$
and 
$$R(x,yc')=
\begin{pmatrix}
R(x)\otimes R'(yc')&0\\
0&-R({}^{c^{-1}}x)\otimes R'(yc')
\end{pmatrix}.$$
Taking the trace, we obtain the following character values
\begin{equation}
\label{eq:charvalselfnonself}
\chi_{V,W}(x,y)=\chi_{V}(x)\chi_W(y),\quad
\chi_{V,W}(x,yc')=(\chi_{V}^+-\chi_{V}^-)(x)\chi_W(yc').
\end{equation}
Note that the matrices $R(xc,y)$ and
$R(xc,yc')$ have zeros on their diagonal, hence
\begin{equation}
\label{eq:charvalselfnonself2}
\chi_{V,W}(xc,y)=0=\chi_{V,W}(xc,yc').
\end{equation}

    By a symmetric approach, if $V$ is non-self-associate and $W$ is
self-conjugate, then $V\widehat \otimes  W$ is also non-self-associate,
and
\begin{equation}
\label{eq:charvalselfnonself3}
\chi_{V,W}(x,y)=\chi_{V}(x)\chi_W(y),\quad
\chi_{V,W}(xc,y)=\chi_V(xc)(\chi_W^+-\chi_W^-)(y),
\end{equation}
and
$$\chi_{V,W}(x,yc')=0=\chi_{V,W}(xc,yc').$$
\smallskip

    Assume that both $V$ and $W$ are self-associate.
Then~(\ref{eq:repself}) still defines an irreducible spin representation
$V\widehat\otimes W$, which is self-associate. We can also determine its
character values more precisely. To do this, we fix a basis
$(w_1,\ldots,w_k)$ of $W^+$, so that $(w_1,\ldots,w_k,c'w_1,\ldots,c'w_k)$
forms a basis of $W$. Writing the matrices $R(xc,y)$ and $R(x,yc')$ (for
$x\in \widetilde{G}_1^+$ and $y\in\widetilde{G}_2^+$) with respect to the
bases $(v_1,\ldots,v_m,cv_1,\ldots,cv_m)$ and
$(w_1,\ldots,w_k,c'w_1,\ldots,c'w_k)$ as above, we observe that these
matrices have zero entries along the diagonal, which proves that
\begin{equation}
\label{eq:charvalselfnonself4}
\chi_{V,W}(x,y)=\chi_{V}(x)\chi_{W}(y),
\end{equation}
and
\begin{equation}
\label{eq:charvalselfnonself5}
\chi_{V,W}(xc,y)=\chi_{V,W}(x,yc')= \chi_{V,W}(xc,yc')= 0.
\end{equation}
\smallskip

    Finally, we assume that both $V$ and $W$ are non-self-associate. Then
the restriction $V_0$ of $V$ to $\widetilde{G}_1^+$ remains irreducible,
and $M=\Ind_{\widetilde G_1^+}^{\widetilde G_1}(V_0)$ is the direct sum of
$V$ and $\varepsilon_{\widetilde G_1}\otimes V$. The vector space $M$ can
also be viewed as the direct sum of $V_0$ and $cV_0$, where $c\in
\widetilde{G}_1\,\backslash\,\widetilde{G}_1^+$.
Following~\cite[p.454]{Humphreys}, we define $V\widehat\otimes W$ as
follows. As vector space, it is given by $M\otimes W$, and the action of
$(\gamma_1,\gamma_2)\in G$ is defined by
$$(\gamma_1,\gamma_2)(m\otimes w)=(-1)^{t(m)s(\gamma)}\gamma_1m\,\otimes\,
\gamma_2w,$$
where $t:M\longrightarrow \{0,1\}$ is the function defined by $t(m)=0$ if
$m\in V_0$ and $t(m)=1$ if $m\in cV_0$.

\begin{remark} 
\label{rk:nonassociate}
    In~\cite[p.454]{Humphreys}, the function $t$ is defined with respect
to the decomposition  $V\oplus\varepsilon\otimes V$, which is inconsistent
with the proof of~\cite[Theorem 2.4]{Humphreys}. We now revisit that proof
using the definition of $t$ provided earlier. Let $(v_1,\ldots,v_m)$ and
$(w_1,\ldots,w_r)$ be bases of $V$ and $W$, respectively.
    Denote by $R$ and $R'$ the matrix representations of $\widetilde G_1$
and $\widetilde G_2$ on $V$ and $W$, with respect to these bases. Then,
$M$ has basis $(v_1,\ldots,v_m, cv_1,\ldots,cv_m)$. For
$(\gamma_1,\gamma_2)\in G$, we denote by $R(\gamma_1,\gamma_2)$ the matrix
representing the action of $(\gamma_1,\gamma_2)$ on $M$.

Let $c'\in\widetilde{G}_2\,\backslash\,\widetilde{G}_2^+$. Let $x\in
\widetilde G_1^+$ and $y\in \widetilde G_2^+$. Then we have
\scriptsize
$$R(x,y)=
\begin{pmatrix}
R(x)\otimes R'(y)&0\\
0&R({}^{c^{-1}}x)\otimes R'(y)
\end{pmatrix},\quad
R(xc,y)=
\begin{pmatrix}
0&R({}^{c^{-1}}x)\otimes R'(y)\\
R(x)\otimes R'(y)&0
\end{pmatrix},
\small
$$
$$R(x,yc')=
\begin{pmatrix}
R(x)\otimes R'(yc')&0\\
0&-R({}^{c^{-1}}x)\otimes R'(yc')
\end{pmatrix},$$
\normalsize
and
\scriptsize
$$
R(xc,yc')=
\begin{pmatrix}
0&-R(xc^2)\otimes R'(yc')\\
R({}^{c^{-1}}x)\otimes R'(yc')&0
\end{pmatrix}.
$$
\normalsize
\end{remark}
It follows that
$$\chi_{V,W}(x,y)=(\chi_V(x)+\chi_V({}^{c^{-1}}x))\chi_W(y).$$
However, since $\chi_V({}^{c^{-1}}x)=\chi_V(x)$, we obtain
\begin{equation}
\label{eq:calvalnon}
\chi_{V,W}(x,y)=2\chi_V(x)\chi_W(y).
\end{equation}
Similarly, $\chi_{V,W}(x,yc')=(\chi_V(x)-\chi_V(x))\chi_W(yc')=0$.
Additionally, using the fact that the diagonal entries of $R(xc,y)$ and
$R(xc,yc')$ are zero, we deduce that
$$\chi_{V,W}(xc,y)=0=\chi_{V,W}(xc,yc').$$

\subsection{Irreducible spin characters of $G^+$.}

    As we see in \S\ref{subsec:index2}, any irreducible non-self-associate
spin representation of $G$ restricts irreducibly to $G^+$, whereas an
irreducible self-associate representation splits into the direct sum of
two irreducible spin representations.

    In the latter case, for an irreducible self-associate representation
$V$ of $G$, with matrix representation $R_V$ and character $\chi_V$, we
now recall a way for computing the difference character
$\Delta(\chi_V)=\chi_V^+-\chi_V^-$. Since $R_V$ and
$\varepsilon_{G}\otimes R_V$ are conjugate, there exists an intertwiner
$T$ such that 
$$\varepsilon_{{G}}\otimes R_V=TR_VT^{-1}.$$ 
    Furthermore, by Schur's Lemma, $T^2$ is a scalar, which we may assume
to be equal to $1$. In other words, $T^2=I$, so $T$ is diagonalizable with
eigenvalues $\pm 1$. Given that $G^+$ is the kernel of
$\varepsilon_{{G}}$, it follows that $T$ commutes with $R(g)$ for all
$g\in G^+$, and therefore the corresponding eigenspaces are then
$G^+$-stable. These eigenspaces are precisely $V^+$ and $V^-$, and the
notation can be chosen so that $V^+$ is the eigenspace corresponding to
the eigenvalue $1$.
Hence, for all $g\in G^+$,
\begin{equation}
\label{eq:diffentrelacement}
\operatorname{Tr(R_V(g)T)}=\chi_V^+(g)-\chi_V^-(g)=\Delta(\chi)(g).
\end{equation}

    We now return to the two types of self-conjugate representations of
$G$ described in the previous paragraph, namely those of the form
$V\,\widehat\otimes\, W$, where $V$ and $W$ are either both self-associate
or both non-self-associate.
\medskip

    First, assume that both $V$ and $W$ are self-associate. We denote by
$\chi_V$ and $\chi_W$ the corresponding characters. Let $c$ and $c'$ be as
above, and let $$e=(v_1,\ldots,v_n,cv_1,\ldots,cv_n)\quad\text{and}\quad
f=(w_1,\ldots,w_r,c'w_1,\ldots,c'w_r)$$ be bases of $V$ and $W$,
respectively, chosen so as to respect the decompositions $V=V^+\oplus V^-$
and $W=W^+\oplus W^-$. For $(\gamma_1,\gamma_2)\in G$, we denote by
$R(\gamma_1,\gamma_2)$ for the matrix representing the action of
$(\gamma_1,\gamma_2)$ on $V\,\widehat \otimes\,W$ with respect to the
basis of $V\otimes W$ induced by the bases $e$ and $f$.
A straightforward computation shows that the matrix
$$J=\begin{pmatrix}
I&0&0&0\\
0&-I&0&0\\
0&0&-I&0\\
0&0&0&I
\end{pmatrix}$$
is an intertwiner of $R$ satisfying $J^2=I$.
Let $x\in \widetilde G_1^+$ and $y\in\widetilde G_2^+$. By
equation~(\ref{eq:diffentrelacement}), we are interested in the values of
$\operatorname{Tr}(R(x,y)J)$ and $\operatorname{Tr}(R(xc,yc')J)$. In the
latter case, since $R(xc,yc')J$ has zeros on the diagonal, the trace is
equal to $0$. In the first case, we can show that
\small
$$R(x,y)=
\begin{pmatrix}
R(x)\otimes R'(y)&0&0&0\\
0&R({}^{c^{-1}}x)\otimes R'(y)&0&0\\
0&0&R(x)\otimes R'({}^{c'^{-1}}y)&0\\
0&0&0&R({}^{c^{-1}}x)\otimes R'(^{c'^{-1}}y)
\end{pmatrix},$$
\normalsize
where $R$ and $R'$ are the matrix representations of the action of
$\widetilde G_1^+$ and $\widetilde G_2^+$ on $V^+$ and $W^+$,
respectively. Then we obtain that $\Delta(\chi_{V,W})$ is zero except for
\begin{align}
\label{eq:valdiffG+self}
\Delta(\chi_{V,W})(x,y)&=\operatorname{Tr}(R(x,y)J)\\
&=\left(\operatorname{Tr}(R(x))-\operatorname{Tr}(R({}^{c^{-1}}x))\right)
\left(\operatorname{Tr}(R'(y))-\operatorname{Tr}(R'({}^{c'^{-1}}y))\right)
\nonumber\\
&=(\chi_V^+(x)-\chi_V^-(x))(\chi_{W}^+(y)-\chi_W^-(y))\nonumber\\
&=\Delta(\chi_V)(x)\Delta(\chi_W)(y).\nonumber
\end{align}

    Now, assume that $V$ and $W$ are non-self-associate. We keep the same
notation as in Remark~\ref{rk:nonassociate}. Using the matrices given
after the remark, we can check that 
$$T=\begin{pmatrix}
0&iR(c)\\
-iR(c)^{-1}&0
\end{pmatrix}\otimes I$$
is an intertwiner satisfying $T^2=I$. If $x\in \widetilde{G}_1^+$ and
$y\in\widetilde{G}_2^+$, then the matrix $R(x,y)T$ has zeros on the
diagonal, and thus, has trace equal to $0$. Moreover, we have
$$R(xc,yc')T=
\begin{pmatrix}
iR(xc)\otimes R'(yc')&0\\
0&iR({}^{c^{-1}}xc)\otimes R'(yc')
\end{pmatrix}.$$
Hence, $\Delta(\chi_{V,W})$ is zero except for the following case:
\begin{align}
\label{eq:valdiffG+nonself}
\Delta(\chi_{V,W})(xc,yc')&=\operatorname{Tr}(R(xc,yc')T)\\
&=i(\chi_V(xc)\chi_W(yc')+\chi_V({}^{c^{-1}}xc)\chi_W(yc'))\nonumber\\
&=2i\chi_V(xc)\chi_W(yc').\nonumber
\end{align}
Here we use that
$$\chi_V({}^{c^{-1}}xc)=\chi_V(c^{-1}xcc)=\chi_V(cc^{-1}xc)=\chi_V(xc),$$
because $\chi_V$ is a class-function.

\section{Galois--Navarro equivariant correspondence}
\label{sec:main}

    Let $p$ be an odd prime number. Let $\lambda$ be a bar-partition, with
$p$-bar Littlewood decomposition $(\lambda_{(\overline
p)},\lambda^{[\overline p]})$. We set $r=|(\lambda_{(\overline p)}|$ and
$w$ such that $|\lambda^{[\overline p]})|=pw$. We then set $n=r+pw$, so in
particular, $n=|\lambda|$.

\subsection{Notation}
\label{subsec:not}
From now on, we define 
\begin{equation}
\label{eq:grouptilde}
G=\tSym_r\widehat \times \tSym_{pw}.
\end{equation}
    In \S\ref{subsec:irrhumphreys}, we observe that an irreducible spin
representation $V\,\widehat\otimes\, W$ of $G$ is self-associate if and
only if either both $V$ and $W$ are  self-associate or both are
non-self-associate. Spin representations of $\tSym_r$ and $\tSym_{pw}$ are
labeled by bar-partitions of $r$ and $pw$, respectively, with a
representation being self-associate if the sign of the corresponding
bar-partition is equal to $1$, and non-self-associate otherwise. 

    Let $\mu$ and $\nu$ be bar-partitions of $r$ and $\nu$ of $pw$,
respectively. 

\subsubsection*{Self-associate spin characters of $G$}

If both $\mu$ and $\nu$ have sign equal to $1$, then
\begin{equation}
\label{eq:zeta1}
\zeta_{\mu,\nu}=\xi_{\mu}\,\widehat\otimes\,\xi_{\nu}
\end{equation}
is an irreducible spin self-associate character of $G$. If both $\mu$ and
$\nu$ have sign $-1$, then 
\begin{equation}
\label{eq:zeta2}
\zeta_{\mu,\nu}=\xi_{\mu}^+\,\widehat\otimes\,\xi_{\nu}^+=
\xi_{\mu}^+\,\widehat\otimes\,\xi_{\nu}^-=
\xi_{\mu}^-\,\widehat\otimes\,\xi_{\nu}^+=
\xi_{\mu}^-\,\widehat\otimes\,\xi_{\nu}^-
\end{equation}
is also a self-associate character of $G$. To see that the four characters
given in the previous equation are equal, we use (\ref{eq:calvalnon}).

\subsubsection*{Non-self-associate characters of $G$}
    When $\mu$ and $\nu$ have opposite signs, then the Humphreys tensor
product of their associated spin characters is non-self-associate. We can
be more precise. 
Suppose that $\sgn(\mu)=-1$ and $\sgn(\nu)=1$. Then the two characters
$\xi_{\mu}^+\widehat\otimes \xi_\mu$ and $\xi_\mu^-\widehat\otimes
\xi_\nu$ are associate. We define
\begin{equation}
\label{eq:zeta3}
\zeta_{\mu,\nu}^+=\xi_{\mu}^+\widehat\otimes \xi_\nu\quad\text{and}\quad
\zeta_{\mu,\nu}^-=\xi_{\mu}^-\widehat\otimes \xi_\nu.
\end{equation}
Similarly, if $\sgn(\mu)=1$ and $\sgn(\nu)=-1$ the two corresponding  associate
characters of $G$ are
\begin{equation}
\label{eq:zeta4}
\zeta_{\mu,\nu}^+=\xi_{\mu}\widehat\otimes \xi_\nu^+\quad\text{and}\quad
\zeta_{\mu,\nu}^-=\xi_{\mu}\widehat\otimes \xi_\nu^-.
\end{equation}

\subsubsection*{Self and non-self-associate characters of $G^+$} 

    According to our earlier notation (see Remark~\ref{rk:notambigu}),
when $\mu$ and $\nu$ have the same sign we denote by $\zeta_{\mu,\nu}^{+}$
and $\zeta_{\mu,\nu}^-$ the two irreducible non-self-associate spin
characters of $G^+$.

    When $\mu$ and $\nu$ have opposite signs, $\zeta_{\mu,\nu}$
corresponds to an irreducible self-associate spin character of $G^+$.

\subsection{Main result}

\begin{lemma}
\label{lem:map}
Let $\lambda$ be a bar-partition as above. Then we have a well-defined
correspondence
\begin{equation}
\label{eq:mapmain}
\xi_{\lambda}\longleftrightarrow \zeta_{\lambda_{(\overline p)},\lambda^{[\overline
p]}}\quad\text{and}\quad
\xi_{\lambda}^{\pm}\longleftrightarrow \zeta_{\lambda_{(\overline p)},\lambda^{[\overline
p]}}^{\pm}.
\end{equation}
\end{lemma}

\begin{remark}
\label{rk:not1}
    For clarity, we recall the conventions from Remark~\ref{rk:notambigu}:
$\xi_{\lambda}$ denotes an irreducible spin character of $\tSym_n$ or
$\tAlt_n$, self-associate if $\lambda \in \mathcal{D}_n^+$ (with
corresponding $\zeta_{\lambda_{(\overline p)}, \lambda^{[\overline p]}}$
self-associate of $G$), and non-self-associate otherwise (with
$\zeta_{\lambda_{(\overline p)}, \lambda^{[\overline p]}}^{\pm}$
non-self-associate of $G$ or $G^+$ accordingly). This notation, though
concise, may be confusing.
\end{remark}

\begin{proof}
The result is a direct consequence of Corollary~\ref{cor:signlitteldec}
and \S\ref{subsec:not}.
\end{proof}

\begin{theorem}
\label{thm:main}
The correspondence described in Equation~(\ref{eq:mapmain}) is equivariant
under the action of the Galois--Navarro group of automorphisms $\mathcal
H$.
\end{theorem}

\begin{proof}
    Let $\lambda$ be a bar-partition with $p$-bar Littlewood decomposition
$(\lambda_{(\overline p)},\lambda^{[\overline p]})$. As previously, we
write $r=|\lambda_{(\overline p)}|$, $pw=|\lambda^{[\overline p]}|$ and
$n=r+pw$.

    First, assume that $\lambda\in\mathcal D_n^-$. Then we have
$\sgn(\lambda)=-1$, and by Corollary~\ref{cor:signlitteldec}, the
partitions $\lambda_{(\overline p)}$ and $\lambda^{[\overline p]}$ have
opposite signs.
\begin{enumerate}[(i)]
\item Let $\xi_{\lambda}$ and $\zeta_{\lambda_{\overline
p},\lambda^{[\overline p]}}$ denote the corresponding spin irreducible
self-associate characters of $\tAlt_n$ and $G^+$, respectively. Recall
that $\xi_{\lambda}$ takes integer values. Furthermore, the elements of
$G^+$ are of the form $(x,y)$, where both $x$ and $y$ with the same
signature. If $x$ and $y$ have negative signature, then
Equation~(\ref{eq:charvalselfnonself2}) implies that $\zeta_{\lambda_{\overline
p},\lambda^{[\overline p]}}(x,y)=0$. If $x$ and $y$ have positive
signature, then Equation~(\ref{eq:charvalselfnonself2}) implies that 
\begin{equation}
\label{eq:valG+}
\zeta_{\lambda_{\overline
p},\lambda^{[\overline p]}}(x,y)=\xi_{\lambda_{\overline
p}}(x)\xi_{\lambda^{[\overline p]}}(y)\in\Z.
\end{equation}
This proves that $\zeta_{\lambda_{\overline
p},\lambda^{[\overline p]}}$ also takes values in $\Z$, and thus, both
$\xi_{\lambda}$ and $\zeta_{\lambda_{\overline
p},\lambda^{[\overline p]}}$ are fixed by $\mathcal H$.
\item Let $\xi_{\lambda}^{\pm}$ and $\zeta_{\lambda_{\overline
p},\lambda^{[\overline p]}}^{\pm}$ denote the spin irreducible
non-self-associate characters of $\tSym_n$ and $G$, respectively. 
Without loss of generality, we assume that $\sgn(\lambda_{(\overline
p)})=1$ and  $\sgn(\lambda^{[\overline
p]})=-1$. The case where $\sgn(\lambda_{(\overline
p)})=-1$ and  $\sgn(\lambda^{[\overline
p]})=1$ follows identically by Equation~(\ref{eq:charvalselfnonself3}). 
Let $f\in\mathcal H$.
Since $\zeta_{\lambda_{\overline
p},\lambda^{[\overline p]}}\in\Irr(G^+)$ is fixed by $f$, \cite[Lemma
2.2]{BrNa} gives that $f$ acts on $\left\{\zeta_{\lambda_{\overline
p},\lambda^{[\overline p]}}^+,\zeta_{\lambda_{\overline
p},\lambda^{[\overline p]}}^-\right\}$. Then we can introduce a sign
$\tau(\zeta_{\lambda_{\overline
p},\lambda^{[\overline p]}},f)$ such that
\begin{equation}
\label{eq:defsignhum}
f(\zeta_{\lambda_{\overline
p},\lambda^{[\overline p]}}^{\pm})=\zeta_{\lambda_{\overline
p},\lambda^{[\overline p]}}^{\pm \tau(\zeta_{\lambda_{\overline
p},\lambda^{[\overline p]}},f)}.
\end{equation}
    Note that, by Equations~(\ref{eq:charvalselfnonself}) and
(\ref{eq:charvalselfnonself2}), the only elements for which, possibly,
$\zeta_{\lambda_{\overline p},\lambda^{[\overline p]}}^{\pm}$ does not
take integer values are those of the form $(x,y)$, with $x$ having
signature $1$ and $y$ having signature $-1$. Moreover, in this case,
Equation~(\ref{eq:charvalselfnonself}) shows that
\begin{align}
\label{eq:relationbarcore}
\tau(\zeta_{\lambda_{\overline
p},\lambda^{[\overline p]}},f)&=\tau(\lambda_{(\overline
p)},f)\tau(\lambda^{[\overline p]},f)\\
\nonumber&=\tau(\lambda,f).
\end{align}
The last equality comes from (i) and (iii) of Theorem~\ref{thm:little}. 
The result follows.
\end{enumerate}

    Now, assume that $\lambda\in\mathcal D_n^+$. Hence, $\sgn(\lambda)=1$,
and $\sgn(\lambda_{(\overline p)})=\sgn(\lambda^{[\overline p]})$ by
Corollary~\ref{cor:signlitteldec}. In this case,
$\xi_{\lambda}\in\Irr(\tSym_n)$ takes values in $\Z$, and the same holds
for $\zeta_{\lambda_{(\overline p)},\lambda^{[\overline p]}}$ by
Equations~(\ref{eq:charvalselfnonself4}) and (\ref{eq:calvalnon}).
Therefore, $\mathcal H$ acts on
$\{\xi_{\lambda}^-,\xi_{\lambda}^+\}\subseteq\Irr(\tAlt_n)$ and also on
$\{\zeta_{\lambda_{(\overline p)},\lambda^{[\overline p]}}^-,
\zeta_{\lambda_{(\overline p)},\lambda^{[\overline
p]}}^+\}\subseteq\Irr(G^+)$.
Let $f\in\mathcal H$. Understanding the action of $f$ on
$\zeta_{\lambda_{(\overline p)},\lambda^{[\overline p]}}^{\pm}$ is
equivalent to understanding its action on
$\Delta(\zeta_{\lambda_{(\overline p)},\lambda^{[\overline p]}})$. As
previously, we set $\tau(\zeta_{\lambda_{(\overline
p)},\lambda^{[\overline p]}},f)=1$ if $f$ fixes
$\Delta(\zeta_{\lambda_{(\overline p)},\lambda^{[\overline p]}})$, and
$-1$ otherwise. 
We now proceed with a more detailed analysis, depending to the signs of
$\lambda_{(\overline p)}$ and $\lambda^{[\overline p]}$.
\begin{enumerate}[(i)]
\item Assume that $\sgn(\lambda_{(\overline p)})=\sgn(\lambda^{[\overline
p]})=1$. The difference character $\Delta(\zeta_{\lambda_{(\overline
p)},\lambda^{[\overline p]}})$ vanishes except on elements of the
form $(x,y)\in G^+$, where $x\in\tAlt_{r}$ and $y\in\tAlt_{pw}$. Moreover,
Equation~(\ref{eq:valdiffG+self}) gives
\begin{align*}
f(\Delta(\zeta_{\lambda_{(\overline p)},\lambda^{[\overline p]}})(x,y))&=
f(\Delta(\xi_{\lambda_{(\overline
p)}})(x))f(\Delta(\xi_{\lambda^{[\overline p]}})(y)),
\end{align*}
hence
\begin{align}
\label{eq:relationbarcore2}
\tau(\zeta_{\lambda_{(\overline
p)},\lambda^{[\overline p]}},f)
&=\tau(\lambda_{(\overline p)},f)\tau(\lambda^{[\overline p]},f)\\
\nonumber&=\tau(\lambda,f)
\end{align}
by (i) and (iii) of Theorem~\ref{thm:little}. It follows that
$f$ commutes with $\xi_{\lambda}^{\pm}\mapsto \zeta_{\lambda_{(\overline
p)},\lambda^{[\overline p]}}^{\pm}$, as required.
\item Assume that $\sgn(\lambda_{(\overline p)})=\sgn(\lambda^{[\overline
p]})=-1$.
The difference character $\Delta(\zeta_{\lambda_{(\overline
p)},\lambda^{[\overline p]}})$ vanishes except on elements of the
form $(x,y)\in G^+$, where $x\in\tSym_r\backslash \tAlt_{r}$ and
$y\in\tSym_{pw}\backslash \tAlt_{pw}$. By
Equation~(\ref{eq:valdiffG+nonself}), we have
\begin{align*}
f(\Delta(\zeta_{\lambda_{(\overline p)},\lambda^{[\overline
p]}})(x,y))&=f(2i)
f(\xi_{\lambda_{(\overline p)}}(x))f(\xi_{\lambda^{[\overline p]}}(y)).
\end{align*}
Then we deduce from Equation~(\ref{eq:Sntildemove}) that
\begin{equation}
\label{eq:movepreuve}
\tau(\zeta_{\lambda_{(\overline
p)},\lambda^{[\overline p]}},f)=\tau(i,f)\tau(\lambda_{(\overline
p)},f)\tau(\lambda^{[\overline p]},f).
\end{equation}
If $f$ acts trivially on the $p'$-roots of unity, then $\tau(i,f)=1$, and
part (iii) of Theorem \ref{thm:little} implies that
$$\tau(\zeta_{\lambda_{(\overline
p)},\lambda^{[\overline p]}},f)=\tau(\lambda,f).$$
If $f=\sigma_p$, then $\tau(i,\sigma_p)=(-1)^{(p-1)/2}$, and by part (ii)
of Theorem \ref{thm:little} and Equation~(\ref{eq:movepreuve}), we obtain
that
$$\tau(\zeta_{\lambda_{(\overline
p)},\lambda^{[\overline p]}},\sigma_p)=\tau(\lambda,\sigma_p).$$
In all cases, this proves that the correspondence 
is $\mathcal H$-equivariant.
\end{enumerate}
\end{proof}

\begin{theorem}
\label{thm:conserveval}
Let $\lambda\in\mathcal D_n$.
Assume that $r=|\lambda_{(\overline p)}|$ is the residue of $n$ modulo
$p$. Then the characters labeled by $\lambda$ and and their corresponding
characters under the map (\ref{eq:mapmain}) have the same $p$-valuation.
\end{theorem}

\begin{proof}
    Let $\lambda$ be a bar-partition. We recall from~\cite[(7.2)]{olsson}
that the degree of $\xi_{\lambda}$ is, up to a power of $2$, equal to
$$\frac{n!}{\prod_{h\in\mathcal H(\lambda)}h},$$ where $\mathcal
H(\lambda)$ is the set of bar-hooklengths of $\lambda$
(see~\cite[p.26]{olsson}). For any integer $N$, we denote by $\nu_p(N)$
the $p$-valuation of $N$. Therefore, we have
\begin{equation}
\label{eq:degree}
\nu_p(\xi_{\lambda}(1))=\nu_p(n!)-\nu_p\left(\prod_{h\in\mathcal
H(\lambda)}h\right).
\end{equation}
Now, we write
$$n=r+a_1p+\cdots+a_rp^r$$ for the $p$-adic decomposition of $n$. Then 
$$pw=a_1p+\cdots+a_rp^r$$ is the $p$-adic decomposition of $pw$. Moreover, 
recall that
$$\nu_p(n!)=\frac{1}{p-1}\left(n-\sum_{j=0}^r a_j\right).$$
    In particular, this gives $\nu_p(n!)=\nu_p((pw)!)$. On the other hand,
\cite[Theorem 4.3]{olsson} shows that the set of bar-hooklengths divisible
by $p$ depends only on the $p$-bar quotient of $\lambda$. Since $\lambda$
and $\lambda^{[\overline p]}$ have the same $p$-bar quotient, they have
the same set of bar-hooklengths divisible by $p$. This implies that
$$\nu_p\left(\prod_{h\in\mathcal
H(\lambda)}h\right)=\nu_p\left(\prod_{h\in\mathcal
H(\lambda^{[\overline p]})}h\right).$$
This proves that
$$\nu_p(\xi_{\lambda}(1))=\nu_p(\xi_{\lambda^{[\overline
p]}}(1)).$$

    Since $r<p$ and $\lambda_{(\overline p)}$ is a $p$-bar core, it
contains no bar hooks divisible by $p$. Hence,
$\nu_p(\xi_{\lambda_{(\overline p)}}(1))=1$. 
Then by Equations~(\ref{eq:charvalselfnonself4}) and~(\ref{eq:calvalnon}),
we obtain that
$$\nu_p(\xi_\lambda(1))=\nu_p(\xi_{\lambda^{[\overline p]}}(1))=
\nu_p(\xi_{\lambda_{(\overline p)}}(1))
\nu_p(\xi_{\lambda^{[\overline p]}}(1))=\nu_p(\zeta_{\lambda_{(\underline
p)},\lambda^{[\underline p]}}(1)).$$
Furthermore, by Clifford theory, since $G^+$ has index two, we also deduce
that
$$
\nu_p(\xi_\lambda^\pm(1))=\nu_p(\xi_\lambda^\pm(1))\quad\text{and}\quad
\nu_p\left(\zeta_{\lambda_{(\overline p)},\lambda^{[\overline p]}}(1)\right)=
\nu_p\left(\zeta_{\lambda_{(\overline p)},
\lambda^{[\overline p]}}^{\pm}(1)\right).$$
The result follows.
\end{proof}

\section{Consequences in Block Theory}
\label{sec:pblocks}

    Let $r\geq 1$ and $w\geq 1$ be positive integers, and $p$ be an odd
prime number. Let $\kappa$ be a $p$-bar core partition of size $r$. We set
$n=pw+r$

\subsection{Blocks of the double covering groups of the symmetric and
alternating groups.}

We denote by 
\begin{equation}
\label{eq:blocktilde}
\B_{\kappa,w} \quad\text{and}\quad \bs_{\kappa,w}
\end{equation}
the spin $p$-blocks of $\tSym_n$ and $\tAlt_n$, respectively, labeled by
$\kappa$ and with bar $p$-weight $w$. More precisely, according to the
\emph{Morris Conjecture} (see~\cite{Humphreys-Blocks}), $\B_{\kappa,w}$
consists of the spin characters of $\tSym_n$ whose labeling bar-partition
$\lambda$ has $p$-bar core $\kappa$ and bar $p$-weight $w_{\overline
p}(\lambda)=w$. Since $\B_{\kappa,w}$ always contains at least one
non-self-associated character, it follow from \cite[Theorem 9.2]{Navarro}
that the set of the constituents obtained by restricting the characters in
$\B_{\kappa,w}$ to $\tAlt_n$, which is precisely $\bs_{\kappa,w}$, forms a
single spin $p$-block of $\tAlt_n$.

\subsection{Blocks of $G$ and $G^+$.}

    We continue to use the notation introduced above, in particular, $G$
denotes the group defined in (\ref{eq:grouptilde}).
We write
\begin{equation}
\label{eq:mamapart}
\mathcal B_w=\{\lambda\mid \lambda_{\overline
p}=\emptyset,\,w_{\overline p}(\lambda)=w\},
\end{equation}
where $w_{\overline p}(\lambda)$ is defined in~(\ref{eq:pweightcar}).
We define the subsets $$B_{\kappa,w}\subseteq \Irr(G)\quad\text{and}\quad
b_{\kappa,w}\subseteq \Irr(G^+)$$ by
\begin{equation}
\label{eq:defBhat}
\left\{ 
\begin{array}{ll}
\zeta_{\kappa, \mu} & \text{if } \zeta_{\kappa, \mu} \text{ is self-associate}, \\
\zeta_{\kappa, \mu}^{+}, \zeta_{\kappa, \mu}^{-} & \text{otherwise}
\end{array}
\;\middle|\; \mu \in \mathcal B_{w} \right\}.
\end{equation}

\begin{remark} Note that, by Remark~\ref{rk:notambigu}, due to our choice
of notation, the characters $\zeta_{\kappa,\mu}$, $\zeta_{\kappa,\mu}^+$
and $\zeta_{\kappa,\mu}^-$ denote either a character of $G$, or a
character of $G^+$ depending on the context. Hence, $B_{\kappa,w}$ and
$b_{\kappa,w}$ are well-defined.
\end{remark}

\begin{theorem}
\label{thm:blockhat}
The sets $B_{\kappa,w}$ and $b_{\kappa,w}$ form a $p$-block of $G$ and
$G^+$, respectively. Moreover:
\begin{enumerate}[(i)]
\item These blocks have the same defect group as the $p$-block 
$B_{\emptyset,w}$.
\item If $\zeta$ is an irreducible character of $G$ (or $G^+$) labeled by
$\kappa$ and $\mu\in \mathcal B_{w}$ as in~(\ref{eq:defBhat}), then
$\zeta$ and the irreducible character
labeled by $\mu$ (in the corresponding $p$-block of $\tSym_{pw}$ or
$\tAlt_{pw}$) have the same $p$-height.
\end{enumerate}
\end{theorem}

\begin{proof}
First, we will consider the case of $B_{\kappa,w}$.
From now on, we may, without loss of generality, replace $G$ 
with its covering groups $\Gamma= \tSym_r\times \tSym_{pw}$, 
equipped with the multiplication defined in
(\ref{eq:loigamma}), since $G$ is a quotient of $\Gamma$. Moreover, as the
covering map $\Gamma\longrightarrow G$ has a $2$-group as kernel and $p$
is odd, the study of $p$-blocks and irreducible characters are unaffected:
there is a natural bijection between the $p$-blocks of $G$ and those
$p$-blocks of $\Gamma$ whose irreducible characters are trivial on the
kernel.

Using the multiplication~(\ref{eq:loigamma}), we check that
$$H=\tAlt_r\times \tAlt_{pw}$$ 
is a normal subgroup of $\Gamma$ with index $4$. Observe that, restricted
to $H$, the group operation is just the standard direct product.
Recall that a $p$-block of $\tSym_r$ labeled by $\kappa$ has defect $0$,
and consists of one irreducible character
$$
\xi_{\kappa} \quad (\text{if $\kappa$ has positive sign}), \quad
\text{or one of } \xi_{\kappa}^{+},\,\xi_{\kappa}^{-} \quad (\text{if
$\kappa$ has negative sign}).
$$
    In the first case, the character $\xi_{\kappa}$ is self-associate, and
its restriction to $\tAlt_r$ splits into two distinct irreducible
characters $\xi_{\kappa}^+$ and $\xi_{\kappa}^-$, which lie in two
separate defect-zero $p$-blocks of $\tAlt_r$. In the second case, the
characters $\xi_\kappa^+$ and $\xi_\kappa^-$ are non-self-associate and
restrict to the same irreducible character $\xi_\kappa$ of $\tAlt_r$,
which lies in a single defect-zero $p$-block.

    Let $b \subseteq \Irr(H)$ be the set of all irreducible characters of
the form $ \chi_1 \otimes \chi_2$, where
$\chi_1\in\{\xi_\kappa,\xi_\kappa^\pm\}$ according to the sign of
$\kappa$,

    Since $H$ is a direct product (and $p$-blocks of such a group are
products of $p$-blocks), we deduce that $b$ is the union of one or two
blocks, namely 
$$\{\xi_\kappa\}\otimes \bs_{\emptyset,w}, \quad\text{or}\quad
(\{\xi_\kappa^+\}\otimes \bs_{\emptyset,w})\sqcup (\{\xi_\kappa^-\}\otimes
\bs_{\emptyset,w}).$$
    Furthermore, (\ref{eq:loigamma}) shows that $b$ is a
$\Gamma$-conjugacy class of blocks of $H$.
    Then using the character values of the characters of $B_{\kappa,w}$
given in equations (\ref{eq:charvalselfnonself}),
(\ref{eq:charvalselfnonself3}), (\ref{eq:charvalselfnonself4}) and
(\ref{eq:calvalnon}), we see that the characters of $B_{\kappa,w}$ are
precisely those characters of $\Gamma$ whose restrictions to $H$ have
constituents in $b$. Therefore, by applying~\cite[Theorem 9.2]{Navarro} to
$\Gamma$ and $H$, we conclude that $B_{\gamma,w}$ is a union of the
$p$-blocks of $\Gamma$ lying above $b$.

    We now show that this union of blocks is, in fact, a single $p$-block.
To establish this, we observe that the $p$-blocks above $b$ form an orbit
under the action of $$\Irr(\Gamma/H)=\langle
\operatorname{sgn}_{\tSym_r},\operatorname{sgn}_{\tSym_{pw}}\rangle,$$
where $\Irr(\Gamma/H)$ acts on $\Irr(\Gamma)$ by tensoring characters. It
therefore suffices to show that $B_{\kappa,w}$ contains a character fixed
by $\Irr(\Gamma/H)$. Before proceeding, we note that the two
bar-partitions with empty $p$-bar core and $p$-bar quotients 
$$((w),\emptyset,\dots,\emptyset)\quad\text{and}\quad
(\emptyset,(w),\emptyset,\dots,\emptyset)$$
have opposite signs. Indeed:

\begin{enumerate}[(i)]
\item Assume that $w$ is even. Then the spin character with empty $p$-bar
core and $p$-bar quotient $((w),\emptyset,\dots,\emptyset)$ has only one
part, so its sign is $-1$. The one with $p$-bar quotient
$(\emptyset,(w),\emptyset,\dots,\emptyset)$ has an even number of parts
$\ell$ (by the ``symmetric pair theory'' developed on page \pageref{brothers}),
and its sign is $(-1)^{pw-\ell}=1$.
\item Assume that $w$ is odd. This time, the same computation shows that
the first character has sign $1$ and the second $-1$.
\end{enumerate}
    In what follows, we denote by $\alpha^+$ the spin character (with
empty $p$-core) among the two described above that has sign $1$, and by
$\alpha^-$ the one with sign $-1$.

    We can now prove that $B_{\kappa,w}$ contains a character fixed by
$\Irr(\Gamma/H)$. Assume first that $\kappa$ has sign $1$. Then
$\zeta_{\kappa, \alpha^+}$ is self-associate. If $\kappa$ has sign $-1$,
then $\zeta_{\kappa,\alpha^-}$ is again self-associate. In either case,
$B_{\kappa,w}$ contains a self-associated character. From the values
computed in equations~(\ref{eq:charvalselfnonself4}),
(\ref{eq:charvalselfnonself5}) and (\ref{eq:calvalnon}), we see that such
a character is invariant under both $\sgn_{\tSym_r}$ and
$\sgn_{\tSym_{pw}}$. This completes the proof of the first part of the
statement.
\smallskip

    Now, since $H$ is a normal subgroup of $\Gamma$ of index prime to $p$,
any $p$-block of $\Gamma$ lying above a $p$-block of $H$ has the same
defect group. In particular, $B_{\kappa,w}$ has the same defect group as
any $p$-block of $H$ contained in $b$, say $b_H$. But $H$ is a direct
product, and the first component of $b_H$ has defect $0$, so its defect
group is trivial. It follows that the defect group of $B_{\kappa,w}$
coincides with that of $\bs_{\emptyset, w}$, which is also the one of
$\B_{\emptyset,w}$ since $[\tSym_{pw},\tAlt_{pw}]=2$ is prime to $p$. 
\smallskip

    Recall that if $\chi$ is an irreducible character, then the $p$-height
$\operatorname{ht}(\chi)$ of $\chi$ is given by
\begin{equation}
\label{eq:height}
\operatorname{ht}(\chi)=\nu_p(\chi(1))+d-\nu_p(|G|),
\end{equation}
where $d$ is the defect of the $p$-block containing $\chi$. Moreover, we
have $\nu_p(\Gamma)=\nu_p(\Gamma^+)=\nu_p(H)$ since $[\Gamma,H]$ is prime
to $p$. The $p$-blocks $B_{\kappa,w}$, $b_{\kappa,w}$ and $b_H$ have the
same defect groups by (i). Furthermore, by Clifford theory \cite[Theorem
6.11]{isaacs}, since $H$ is normal in $\Gamma$ with index prime to $p$,
the $p$-part of the degree of an irreducible character of $\Gamma$ is
equal to that of any irreducible constituents upon restriction to $H$.
Hence, for any $\zeta_{\kappa,\mu}\in B_{\kappa,w}$, if
$\zeta_{\kappa,\mu}'$ denotes one such constituent, then 
\begin{align}
\label{eq:calpreuveheight}
\operatorname{ht}(\zeta_{\kappa,\mu})&=\nu_p(\zeta_{\kappa,\mu}(1))+d-\nu_p(|\Gamma|)\\
\nonumber&=\nu_p(\zeta'_{\kappa,\mu}(1))+d-\nu_p(|H|)\\
\nonumber&=\operatorname{ht}(\zeta'_{\kappa,\nu})\\
\nonumber&=\operatorname{ht}(\xi_{\nu}).
\end{align}
    The last equality follows from the fact that $H=\tAlt_r \times
\tAlt_{pw}$ is a direct product, so the $p$-height of a character is the
sum of the $p$-heights of its components, together with the fact that
$\kappa$ labels a defect zero $p$-block of $\tAlt_r$, hence corresponds to
a character of $p$-height $0$. 

    Note that equation~(\ref{eq:calpreuveheight}) is unambiguous, as
$\operatorname{ht}(\xi_\nu) = \operatorname{ht}(\xi_\nu^\pm)$ holds
whether these characters are considered as characters of $\tSym_{pw}$ or
in $\tAlt_{pw}$, since the index $[\tSym_{pw} : \tAlt_{pw}] = 2$ is prime
to $p$. 
\smallskip

    Now, we consider the case of $\bs_{\kappa,w}$. Since $[G,G^+]=2$ is
prime to $p$, the constituents appearing in the restriction of characters
from $\B_{\kappa,w}$ to $G^+$ form a union of $p$-blocks of $G^+$.
Moreover, if $\sgn(\kappa)=1$, then $\zeta_{\kappa,\alpha^-}$ is not
self-associate. If $\sgn(\kappa)=-1$, then there are two associate
characters of $\B_{\kappa,w}$ labeled by $\kappa$ and $\alpha^+$. In both
cases $B_{\kappa,w}$ contains at least one non-self-associate character.
It then follows from~\cite[Theorem 9.2]{Navarro} that $\bs_{\kappa,w}$ is
a single $p$-block of $G^+$.
The points (i) and (ii) follow from the fact that the result holds for
the caracters of $\B_{\kappa,w}$, and that $[G:G^+]=2$ is prime to $p$.
\end{proof}

\begin{example}
    Assume $w=1$. First, we describe $\mathcal B_1$. This set contains
$(p+1)/2$ bar-partitions, characterized by having $p$-bar quotient with
exactely one non-empty component equal to $(1)$. More precisely, these
partitions are of the form
$$\mu_k=(p-k,k),$$
for integers $0\leq k\leq\frac{p-1}{2}$. Moreover, we have
$$\sgn(\mu_k)=\left\{\begin{array}{l}
1\quad\text{if }k=0\\
-1\quad\text{otherwise.}
\end{array}
\right.$$
Now, if the sign of $\kappa$ is $1$, then the $p$-block of $G$
$$B_{\kappa,1}=\{\zeta_{\kappa,\mu_0},\zeta_{\kappa,\mu_k}^{\pm},\ 1\leq k\leq
(p-1)/2\}$$
contains $p$ spin characters, and the $p$-block of $G^+$
$$b_{\kappa,1}=\{\zeta_{\kappa,\mu_0}^{\pm},\zeta_{\kappa,\mu_k},\ 1\leq k\leq
(p-1)/2\}$$
contains $\frac 1 2(p+3)$ spin characters.
If $\kappa$ is of sign $-1$, we obtain the same result, swapping
$B_{\kappa,1}$ and $b_{\kappa,1}$.
In particular, if $\kappa$ and $\kappa'$ have opposite signs, then it is
clear that $B_{\kappa,1}$ et $b_{\kappa',1}$ have the same cardinality. 
\end{example}

\subsection{Bijections between $p$-blocks}

    The notation introduced above will be used throughout. We define a map
$\Phi$, by setting, for  any bar-partition $\lambda$ of $n$ with $p$-bar
core $\kappa$, 
$$\Phi(\xi_{\lambda})=\zeta_{\kappa,\lambda^{[\overline
p]}}\quad\text{and}\quad
\Phi(\xi_{\lambda}^{\pm})=\zeta_{\kappa,\lambda^{[\overline p]}}^{\pm}.$$
Note that, by our choice of notation, $\Phi$ can be viewed either as a map
between the two $p$-blocks $\B_{\kappa,w}$ and $B_{\kappa,w}$, or between
the two $p$-blocks $\bs_{\kappa,w}$ and $b_{\kappa,w}$.

\begin{theorem} 
\label{thm:bijblockstildehump}
    The map $\Phi$ defines a bijection between $\B_{\kappa,w}$ and
$B_{\kappa,w}$ (respectively, between $\bs_{\kappa,w}$ and
$b_{\kappa,w}$), which is $\mathcal H$-equivariant, height-preserving, and
preserves the defect of the corresponding characters.
\end{theorem}

\begin{proof}
    The characters of $\B_{\kappa,w}$ and $B_{\kappa,w}$ share the same
set of labels. Moreover,~by (\ref{eq:mapmain}), the map $\Phi$ preserves
the property of being associate or non-associate. Hence, $\Phi$ is
bijective. The same holds for the blocks $\bs_{\kappa,w}$ and
$b_{\kappa,w}$. It follows immediately from Theorem~\ref{thm:main} that
the map $\Phi$ is $\mathcal{H}$-equivariant.

    Now, we prove that $\Phi$ preserves the $p$-height of characters. To
this end, we use~\cite[Proposition (13.8)]{olsson}, which states that if
$\lambda$ is a bar-partition with $p$-bar core $\kappa$ and $p$-bar
quotient $\lambda^{(\overline{p})}$, then the $p$-height of the
character(s) labeled by $\lambda$ (whether one or two depending on
associativity) depends only on $\lambda^{(\overline{p})}$.

    In particular, applying the formula to $\lambda^{[\overline p]}$,
which has empty $p$-bar core and the same $p$-bar quotient
$\lambda^{(\overline p)}$, we deduce that the $p$-height of the
character(s) of $\tSym_n$ labeled by $\lambda$ equals the $p$-height of
the character(s) of $\tSym_{pw}$ labeled by $\lambda^{[\overline p]}$. We
then conclude by point (ii) of Theorem~\ref{thm:blockhat}.

    As $\bs_{\kappa,w}$ and $b_{\kappa,w}$ are $p$-blocks of index $2$
subgroups with $p$ odd, Clifford theory ensures that the $p$-height
preservation established for $\B_{\kappa,w}$ and $B_{\kappa,w}$ descends
to these $p$-blocks as well.

    Finally, using~\cite[Lemma (13.2)]{olsson} together with point (i) of
Theorem~\ref{thm:blockhat}, we conclude that the two $p$-blocks share the
same $p$-defect. Since, moreover, $\Phi$ preserves the $p$-height, the
defect of each character is also preserved by $\Phi$.
\end{proof}

\subsection{Bijections between $p$-blocks of covering groups of the
symmetric and alternating groups}

    Let $\kappa$ and $\kappa'$ be two $p$-bar core partitions of size $r$
and $r'$, respectively. Let $w$ be a positive integer, and $p$ be a prime
number. We set
$$n=pw+r\quad\text{and}\quad n'=pw+r'.$$
    We denote by $\Psi$ the map that associates to a bar-partition
$\lambda$ of $n$ (with $p$-bar core $\kappa$ and bar $p$-weight $w$), the
bar-partition $\Psi(\lambda)$ of $n'$, whose Littlewood map is given by
$$(\kappa',\lambda^{(\overline p)}),$$
where $\lambda^{(\overline p)}$ is the $p$-bar quotient of $\lambda$.

\subsubsection{The non-crossing case}

\begin{proposition}
\label{prop:defpsi}
    Assume that $\kappa$ and $\kappa'$ have the same sign. Then the map
$\Psi$ induces a bijection between the $p$-blocks $\B_{\kappa,w}$ and
$\B_{\kappa',w}$ of $\tSym_n$ and $\tSym_{n'}$, and between the two
$p$-blocks $\bs_{\kappa,w}$ and $\bs_{\kappa',w}$ of $\tAlt_n$ and
$\tAlt_{n'}$.
\end{proposition}

\begin{proof}
    A bar-partition $\lambda$ with $p$-bar core $\kappa$ and $p$-bar
quotient $\lambda^{(\overline p)}$ labels either a single character in
$\B_{\kappa,w}$ if $\sgn(\lambda)=1$, or a pair of associate characters if
$\sgn(\lambda)=-1$. The same statement holds for the block $\bs_{\kappa,
w}$. By Corollary~\ref{cor:signlitteldec}, we have
$$\sgn(\lambda)=\sgn(\kappa)\sgn(\lambda^{[\,\overline p\,
]})=\sgn(\kappa')\sgn(\lambda^{[\,\overline p\,]})=\sgn(\Psi(\lambda)),$$ 
since $\sgn(\kappa)=\sgn(\kappa')$ by assumption. The result then follows.
\end{proof}

\begin{remark}
\label{rk:defPsi}
    The maps defined in the previous proposition are compatible with
Clifford theory between $\tSym_n$ and $\tAlt_n$, and $\tSym_{n'}$ and
$\tAlt_{n'}$. We denote these maps by the same symbol $\Psi$ in both
settings, for simplicity.
\end{remark}

\begin{theorem}
\label{thm:blocksmemesigne}
We assume that $\sgn(\kappa)=\sgn(\kappa')$, and that
$$\tau(\kappa,\sigma_p)=\tau(\kappa',\sigma_p),$$ 
where $\tau(\kappa,\sigma_p)$ is defined in~\S\ref{subsec:actionspin}.
Then the maps defined in Proposition~\ref{prop:defpsi} are $\mathcal
H$-equivariant, and preserve both the $p$-height and the defect of each
character and its image under $\Psi$.
\end{theorem}

\begin{proof}
    Denote by $\Phi$ the map defined in Theorem~\ref{thm:bijblockstildehump} 
between the blocks $\B_{\kappa,w}$ and $B_{\kappa,w}$. Denote by $\Phi'$
the analog map between the blocks $\B_{\kappa',w}$ and $B_{\kappa',w}$.
Define
$\vartheta:B_{\kappa,w}\longrightarrow B_{\kappa',w}$ by setting
$$\vartheta(\zeta_{\kappa,\mu})=\zeta_{\kappa',\mu}\quad \text{if }
\zeta_{\kappa,\mu} \text{ is self-associated,}$$
and
$$\vartheta(\zeta_{\kappa,\mu}^\pm)=\zeta_{\kappa',\mu}^\pm\quad
\text{otherwise.}$$
    Note that this map is well-defined, since $\kappa$ and $\kappa'$ have
the same sign. Moreover, by Equation~(\ref{eq:relationbarcore}) and the
subsequent discussion, it follows that $\vartheta$ is $\mathcal
H$-equivariant since $\tau(\kappa,\sigma_p)=\tau(\kappa',\sigma_p)$. Then,
Theorem~\ref{thm:blockhat} implies that $\vartheta$ preserves the defect
group of the $p$-block, as well as the $p$-height and the defect of
characters.

Now, observe that
$$\Psi=\Phi'^{-1}\circ \vartheta\circ \Phi.$$
We conclude using Theorem~\ref{thm:bijblockstildehump}. A similar argument
applies to the blocks $\bs_{\kappa,w}$ and $\bs_{\kappa',w}$.
\end{proof}

\begin{remark}
\label{rk:perfectdirect}
    In~\cite[Theorem 4.21, Corollary 4.22]{BrGr3}, it is proved that under
the assumptions of Theorem~\ref{thm:blocksmemesigne}, the $p$-blocks
$\B_{\kappa,w}$ and $\B_{\kappa',w}$, as well as $\bs_{\kappa,w}$ and
$\bs_{\kappa',w}$, are perfectly isometric. A perfect isometry is
explicitly constructed, and we observe that, up to a sign, this is exactly
the map $\Psi$ defined in Remark~\ref{rk:defPsi}. Since any sign is fixed
by any Galois automorphisms, Theorem~\ref{thm:blocksmemesigne} implies
that the perfect isometries 
$$I:\B_{\kappa,w}\longrightarrow \B_{\kappa',w}\quad\text{and}\quad I:\
\bs_{\kappa,w}\longrightarrow \bs_{\kappa',w}$$
constructed in \cite{BrGr3} are $\mathcal H$-equivariant.
\end{remark}

\begin{remark} 
    Assume that $0\leq r\leq p-1$ and $0\leq r'\leq p-1$. Then, by
Theorem~\ref{thm:conserveval}, the degrees of any character in
$\B_{\kappa,w}$ (or in $\bs_{\kappa,w}$) and its image by $\Psi$ have the
same $p$-valuation. Since the spin $p'$-characters of $\tSym_n$ (or
$\tAlt_n$) lie in such blocks (namely, the characters of $p$-height $0$ in
this case), this result is in line with Navarro's conjecture, which
predicts, at the global level, the existence of an
$\mathcal{H}$-equivariant bijection between the sets of spin
$p'$-characters of the groups $\tSym_n$ and $\tSym_{n'}$.

    Furthermore, Theorem~\ref{thm:conserveval} shows that this
compatibility extends beyond the $p'$-case: the map $\Psi$ preserves the
$p$-valuation of character degrees for all spin characters in the relevant
blocks, and commutes with the action of $\mathcal H$. This provides a
Navarro-type extension, at the global level, in the context of double
covering groups.
\end{remark}
\subsubsection{The crossing case}

\begin{proposition}
\label{prop:defpsi2}
    Assume that $\kappa$ and $\kappa'$ have opposite signs. Then the map
$\Psi$ induces a bijection between the $p$-blocks $\B_{\kappa,w}$ and
$\bs_{\kappa',w}$ of $\tSym_n$ and $\tAlt_{n'}$.
\end{proposition}

\begin{proof}
    Let $\lambda$ be a bar-partition labeling character(s) in
$\B_{\kappa,w}$. As in the proof of Proposition~\ref{prop:defpsi}, we
have
$$\sgn(\lambda)=-\sgn(\Psi(\lambda)).$$ 
Suppose first that $\sgn(\lambda)=1$. Then $\lambda$ labels a single
character of $\B_{\kappa,w}$. Since $\sgn(\Psi(\lambda))=-1$, the bar
partition $\Psi(\lambda)$ labels two associate spin characters of
$\tSym_{n'}$, which restrict to the same character of $\tAlt_{n'}$ by
Clifford theory. According to our notation, this character is denoted by
$\xi_{\Psi(\lambda)}\in \Irr(\tAlt_{n'})$.

    Now, suppose that $\sgn(\lambda)=-1$. Then $\lambda$ labels two
associate spin characters of $\B_{\kappa,w}$. The image $\Psi(\lambda)$
has sign $1$, and hence labels a self-associate character of
$\tSym_{n'}$. By Clifford theory, this character restricts to two
associate characters of $\tAlt_{n'}$. This proves the result.
\end{proof}

\begin{theorem}
\label{thm:crossing}
Assume that $\sgn(\kappa)=-1$ and $\sgn(\kappa'=1)$. Assume that 
$$\tau(\kappa,\sigma_p)=\tau(\kappa',\sigma_p).$$ Then the
map defined in Proposition~\ref{prop:defpsi2} is $\mathcal H$-equivariant,
and preserves both the $p$-height and the defect of each character and its
image under $\Psi$.
\end{theorem}

\begin{proof}
    Let $\Phi$ be the map defined in Theorem~\ref{thm:bijblockstildehump}
between the blocks $\B_{\kappa,w}$ and $B_{\kappa,w}$, and let $\Phi'$
denote the analogous map between the blocks $\bs_{\kappa',w}$ and
$b_{\kappa',w}$. Define a map $\vartheta:B_{\kappa,w}\longrightarrow
b_{\kappa',w}$ by the same formula as in the proof of
Theorem~\ref{thm:blocksmemesigne}.
    This map is well-defined, since $\kappa$ and $\kappa'$ have opposite
signs. Furthermore, by Theorem~\ref{thm:blockhat}, $\vartheta$ preserves
the defect group of the $p$-block, as well as the $p$-height and the
defect of characters. 

    Note first that the assumption
$\tau(\kappa,\sigma_p)=\tau(\kappa',\sigma_p)$ is equivalent to 
$\tau(\kappa,f)=\tau(\kappa',f)$ for any $f\in\mathcal H$.

    Let $\lambda$ be a bar-partition. Suppose first that $\lambda$ labels
a self-associate character of $\B_{\kappa,w}$. In particular,
$\sgn(\lambda)=1$, and by Corollary~\ref{cor:signlitteldec} we have
$\sgn(\lambda^{[\,\overline p\,]})=-1$ since $\sgn(\kappa)=-1$. In this
case, the proof of Theorem~\ref{thm:main} shows that 
$$\tau(\zeta_{\kappa,\lambda^{[,\overline
p\,]}},f)=1.$$ 
    Moreover, $\sgn(\Psi(\lambda))=-1$, and Equation~(\ref{eq:valG+})
implies that 
$$\tau(\zeta_{\kappa',\lambda^{[\,\overline p\,]}},f)=1.$$

    Now suppose that $\lambda$ has sign $-1$, so that $\lambda$ labels two
associate characters of $\B_{\kappa,w}$. Then, by
Corollary~\ref{cor:signlitteldec}, we have $\sgn(\lambda^{[\overline
p]})=1$ . Then
$$\sgn(\kappa')=1=\sgn(\lambda^{[\overline p]}),$$ 
and it follows from Equations~(\ref{eq:relationbarcore}) and
(\ref{eq:relationbarcore2}) that
$$\tau(\zeta_{\kappa',\lambda^{[\overline
p]}},f)=\tau(\kappa',f)\tau(\lambda^{[\overline
p]},f)=\tau(\kappa,f)\tau(\lambda^{[\overline
p]},f)=\tau(\zeta_{\kappa,\lambda^{[\overline p]}},f),$$
since $\tau(\kappa,f)=\tau(\kappa',f)$.
Hence, the map $\varphi$ is $\mathcal H$-equivariant. Finally, since
$\Psi=\Phi'^{-1}\circ\varphi\circ \Phi$, the result follows.
\end{proof}

\begin{remark}
\label{rk:perfectcross}
    Theorem~\ref{thm:crossing} also admits an interpretation in terms of
perfect isometries. Indeed, by \cite[Theorem 4.21]{BrGr3}, the $p$-blocks
$B_{\kappa,w}$ and $b_{\kappa',w}$ are perfectly isometric. Moreover, the
perfect isometry constructed there coincides with the map defined in
Theorem~\ref{thm:crossing}, up to a sign. As a consequence, the perfect
isometry of \cite{BrGr3} (constructed under the assumption $\sgn(\kappa) =
-1 = -\sgn(\kappa')$) is $\mathcal{H}$-equivariant.
\end{remark}

\begin{remark}
\label{rk:fails}
    Note that the Kessar–Schaps crossing conjecture, whose weaker form is
proved in \cite{BrGr3}, also includes the cases where $\sgn(\kappa) = 1$
and $\sgn(\kappa') = -1$. Surprisingly, the perfect isometry constructed
in \cite{BrGr3} in this setting does not commute with the action of
$\mathcal H$. The issue arises when $\Psi(\lambda)$ labels two associate
characters of $\tAlt_n$. In this case, the sign $(-1)^{(p-1)/2}$ appears
in Formula~(\ref{eq:movepreuve}), and it is not compatible with the
formula given in (\ref{eq:relationbarcore}).
\end{remark}

\subsection{The case of the non-spin $p$-blocks}
\label{subsec:nonspin}

    In what follows, $p$ continues to denote an odd prime number. 
We identify the non-spin irreducible characters of $\tSym_n$ and $\tAlt_n$
with those of $\Sym_n$ and its index-two subgroup $\Alt_n$, respectively.
Recall that the irreducible characters of $\Sym_n$ are parametrized by the
partitions of $n$, and we write $\chi_{\lambda}$ for the character labeled
by $\lambda$. The sign character $\varepsilon$ of $\Sym_n$ acts on
$\Irr(\Sym_n)$, partitioning it into orbits of size either one or two. By
Clifford theory (see Subsection~\ref{subsec:index2}), these orbits
correspond to orbits of the conjugation action of $\Sym_n / \Alt_n$ on
$\Irr(\Alt_n)$. 
    More precisely, an orbit of size two in $\Irr(\Sym_n)$ corresponds to
a single irreducible character of $\Alt_n$ obtained by restricting a
character from $\Sym_n$. An orbit of size one in
$\Irr(\Sym_n)$ corresponds to a pair of distinct, conjugate irreducible
characters of $\Alt_n$ whose induction recovers the character in $\Sym_n$.

    It is well known that size-one orbits in $\Irr(\Sym_n)$ are precisely
those labeled by self-conjugate partitions of $n$. For such a partition
$\lambda$, the associated function $\Delta_{\lambda}$ is supported only on
the conjugacy classes of permutations whose cycle type corresponds to the
diagonal hooks of $\lambda$. In this setting, we obtain a formula
analogous to~\eqref{eq:diffAntilte}, where the part lengths are replaced
with diagonal hook lengths; see~\cite[Theorem 2.5.13]{James-Kerber}.
\smallskip

    As seen in Section~\ref{subsec:frob}, the diagonal hooks of a
partition $\lambda$ can be interpreted via the pointed abacus of
$\lambda$. More precisely, there is a bijection between the set of arms, or
equivalently, the set of black beads above the fence in the pointed
abacus, and the diagonal hooks of $\lambda$. Furthermore, if $\lambda$ is
self-conjugate, the length $d_x$ of the diagonal hook corresponding to the
arm in position $x$ is given by
$$d_x=2x+1.$$

    In fact, one can associate to $\lambda$ a pointed $p$-abacus, which
allows us to define the \emph{$p$-characteristic vector} of $\lambda$, its
$p$-core $\lambda_{(p)}$, and its $p$-quotient $\lambda^{(p)}$. A
partition with empty $p$-core is called a \emph{$p$-cocore partition}. The
$p$-cocore partition associated with $\lambda^{(p)}$ is denoted
$\lambda^{[p]}$. We refer to~\cite[\S3.1]{BrNa2} for further details.
\medskip

    This leads to a complete analogy at the combinatorial level between
self-conjugate partitions and bar-partitions, with diagonal hooks playing
the role of parts.  
    As a concrete illustration, we now present a self-conjugate analog
of Equation~(\ref{eq:longeurrelation}). First, observe that the integer
$d$ appearing in~(\ref{eq:longeurrelation}) has a natural combinatorial
interpretation. This is the number of parts that ``disappear'' when
passing from the $p$-bar core and $p$-bar cocore to the full partition via
the push-and-pull process. A direct analog exists in the case of
self-conjugate partitions, namely the integer $d$ can be interpreted as the
number of diagonal hooks that disappear when going from the $p$-core and
$p$-cocore to the partition through the push-and-pull process described
in~\cite[Theorem 3.27]{BrNa2}.

    Therefore, if we replace the bar-length $\ell(\lambda)$ with the
Durfee number $\mathfrak c(\lambda)$, that is, the number of diagonal
hooks in $\lambda$, and follow the proof of
Lemma~\ref{lemma:longueurlittelwoodec} with the same strategy, we obtain:

\begin{equation}
\label{eq:durfeewoordec}
\mathfrak{c}(\lambda)=  \mathfrak{c}(\lambda_{(\overline t)}) +
\mathfrak{c}(\lambda^{[\overline t]}) -2d.
\end{equation}
\medskip

    In the spin setting, the notion of bar-abacus pairing is defined for
$p$-bar cocore partitions, via their $p$-bar-abacus. A pair of paired
beads then corresponds to parts $d_x$ and $d_{x^*}$ of the bar-partition
$\lambda$. The key relation satisfied by these pairs, which
plays a crucial role in the proof of Theorem~\ref{thm:little}, is that
\begin{equation}
\label{eq:fondbro}
d_x + d_{x^*} \equiv 0 \mod p.
\end{equation}

    In the case of self-conjugate partitions, by interpreting the beads on
the abacus as diagonal hooks of a self-conjugate $p$-cocore partition, we
can also develop an \emph{abacus pairing theory} for $p$-cocore
partitions, in a manner analogous to that of Section~\ref{subsection:pairing}. 

    More precisely, if $\lambda$ is a self-conjugate $p$-cocore partition,
then its $p$-quotient has the form $$\lambda^{(p)} = \left(\lambda^0,
\ldots, \lambda^{(p-1)/2}, \lambda^{(p-1)/2*}, \ldots, \lambda^{0*}
\right),$$ where $\lambda^*$ denotes the conjugate partition of $\lambda$.
We observe that all the information of $\lambda$ is encoded in the first
$(p+1)/2$ components of its $p$-quotient. Since $\lambda$ is a $p$-cocore,
the abacus of each $\lambda^j$ is pointed. Moreover, for $0\leq j\leq
(p-3)/2$, if $b$ is a white bead at position $x$ below the fence in the
abacus of $\lambda^j$, then there is a corresponding black bead at
position $x$  above the fence in the abacus of $\lambda^{p-1-j}$. This
allows us to associate to each arm $x$ in the abacus of $\lambda^j$ its
corresponding leg (as the abacus is pointed), which in turn corresponds to
an arm in the abacus of $\lambda^{p-1-j}$. Since the arms parametrize the
diagonal hooks of $\lambda$, this construction yields paired diagonal
hooks.

    Now, given a pair of pairing beads $(x,x^*)$ on runners $i$ and $p-1-i$,
the corresponding diagonal hook of the $p$-cocore partition are given
by
$$
d_x = 2(px + i) + 1, \quad d_{x^*} = 2 \bigl(p x^* + (p - 1 - i)\bigr) + 1,
$$
and they satisfies the same congruence as in equation~(\ref{eq:fondbro}).
    Hence, using equation~(\ref{eq:durfeewoordec}) together with the
abacus pairing for self-conjugate $p$-cocore partitions described above,
the mechanism of the proof applies in the same way, and shows
that for any $f \in \mathcal{H}$, and for any self-conjugate partition
$\lambda$, one has 

\begin{equation}
\label{eq:taunonspin}
\tau(\lambda,f)=\tau(\lambda_{(p)},f)\tau(\lambda^{[p]},f),
\end{equation}
where $\tau(\lambda,f)$ is defined analogously, as in
Equation~(\ref{eq:Antildemove}).

    This factorization, analogous to that obtained in the spin case,
reduces the study of the Galois action to the $p$-core and
the $p$-quotient of $\lambda$.

\medskip

    We now consider the parametrization of non-spin characters of $G$.
These are in bijection with the irreducible characters of 
$$G/\langle z
\rangle \cong \Sym_{r}\times \Sym_{pw},$$
and may therefore be naturally
identified with characters of the form 
$$
\eta_{\lambda_{(p)},\lambda^{[p]}}=\chi_{\lambda_{(p)}} \otimes \chi_{\lambda^{[p]}}, 
$$
    where $\lambda$ is a partition of $n$, and $\lambda_{(p)}$ and
$\lambda^{[p]}$ denote its $p$-core and $p$-cocore, respectively.
Moreover, if $\lambda_{(p)}$ is self-conjugate, then
$\eta_{\lambda_{(p)},\lambda^{[p]}}$ is self-associate if and only if
$\chi_{\lambda^{[p]}}$ is. In addition, $\lambda$ is self-conjugate if and
only if $\lambda_{(p)}$ is. In particular, for such a $p$-core, the
correspondence between the characters of $\Alt_n$ and those of $G^+$ given
by
$$\chi_{\lambda}\longleftrightarrow \eta_{\lambda_{(p)},\lambda^{[p]}}\quad
\text{and}\quad
\chi_{\lambda}^{\pm}\longleftrightarrow
\eta_{\lambda_{(p)},\lambda^{[p]}}^{\pm}$$
is well-defined. For the notational conventions, we refer to
\S\ref{subsec:index2}.
\medskip

    Let $\kappa$ be a symmetric $p$-core partition, and let $w$ be a
positive integer. Set $n=|\kappa|+pw$. By \emph{Nakayama's conjecture}
(now a theorem, see~\cite[\S6.1.21]{James-Kerber}), the irreducible characters
$\chi_{\lambda}$ of $\Sym_n$ with $p$-core $\kappa$ and $p$-weight
$w$, that is, such that $w=(|\lambda|-|\kappa|)/p$, belong to a single $p$-block. 
Moreover, by~\cite[Theorem 9.2]{Navarro}, the restrictions of these
characters to $\Alt_n$ also belong to a single $p$-block of $\Alt_n$, which we
denote by $b_{\kappa,w}$.
    On the other hand, the irreducible non-spin characters of $G^+$
labeled by pairs $(\kappa, \nu)$, where $\nu$ is a partition with empty
$p$-core and $p$-weight $w$, form a $p$-block of $G^+$, denoted by
$b'_{\kappa,w}$.

    The correspondence described above induces a bijection between the
blocks $b_{\kappa,w}$ and $b'_{\kappa,w}$, which preserves the defect
groups and the $p$-height of characters by \cite[(11.2) and
(11.5)]{olsson}, since $p$ is odd and coprime to
$[\Sym_n:\Alt_n]=2=[G:G^+]$. Using Equation~\eqref{eq:taunonspin}, we
further deduce that this bijection is $\mathcal H$-equivariant. \medskip

    We now turn our attention to the perfect isometries between non-spin
blocks of equal $p$-weight, constructed by Enguehard for symmetric
groups~\cite{Enguehard} and by the first author and J.\,B. Gramain for alternating
groups~\cite{BrGr3}. Our focus is on the $\mathcal H$-equivariance of
these isometries. Since the character values of symmetric groups are
integral, the only cases relevant to this question arise from blocks of
alternating groups that contain associated characters, that is, blocks
parametrized by self-conjugate $p$-cores. We will therefore concentrate on
these cases.

    Let $\kappa$ and $\kappa'$ be two self-conjugate $p$-cores, and let $w$
be a positive integer. Then the map
$$\Psi:b_{\kappa,w}\longrightarrow b_{\kappa',w}$$ 
that associates to $\chi_{\lambda}$ (or $\chi_{\lambda}^{\pm}))$ the
character $\chi_{\lambda'}$ (or $\chi_{\lambda'}^{\pm}$) where $\lambda'$
is the partition of $|\kappa'|+pw$ with $p$-core $\kappa'$ and the same
$p$-quotient as $\lambda$, is well-defined. As in the spin case, we deduce
the following result.

\begin{theorem}
\label{thm:nonspin}
 Assume that
\begin{equation}
\label{eq:condalt}
\tau(\kappa,\sigma_p)=\tau(\kappa',\sigma_p).
\end{equation}
Then $\Psi$ is $\mathcal H$-equivariant, and it preserves both the
$p$-height and the defect of each character and its image under $\Psi$.
\end{theorem}

\begin{remark}
\label{rk:isoalt}
    As in Remark~\ref{rk:perfectdirect}, Theorem~\ref{thm:nonspin} implies
that the perfect isometries constructed in \cite[Theorem 3.9]{BrGr3} are
$\mathcal H$-equivariant whenever the labeling $p$-cores have the same
rationality, that is, when Condition~(\ref{eq:condalt}) is satisfied.
\end{remark}

{\bf Acknowledgements.} The first author acknowledges the support of the
ANR project Cortipom ANR-21-CE40-0019. The second author acknowledges the
support of PSC-CUNY TRADA-47-785. 
The authors would like to thank the American Institute of Mathematics
(AIM) for its generous support and hospitality during a SQuaRE program,
where substantial progress on this project was made. The authors wish to
thank  Christopher Hanusa for suggesting the terminology of $p$-cocore and
$p$-bar cocore partitions introduced in this paper.

\bibliographystyle{abbrv}
\bibliography{references_19_1}

\begin{thebibliography}{10}

\bibitem{Broue}
M.~Brou{\'e}.
\newblock Isom\'etries parfaites, types de blocs, cat\'egories d\'eriv\'ees.
\newblock {\em Ast\'erisque}, (181-182):61--92, 1990.

\bibitem{BrGr3}
O.~Brunat and J.-B. Gramain.
\newblock Perfect isometries and {M}urnaghan-{N}akayama rules.
\newblock {\em Trans. Amer. Math. Soc.}, 369(11):7657--7718, 2017.

\bibitem{BrNa2}
O.~Brunat and R.~Nath.
\newblock Cores and quotients of partitions through the {F}robenius symbol.
\newblock {\em preprint}, 2020.

\bibitem{BrNa}
O.~Brunat and R.~Nath.
\newblock The {N}avarro {C}onjecture for the alternating groups.
\newblock {\em Algebra and Number Theory}, 15(4):821--862, 2021.

\bibitem{CabSpath}
M.~Cabanes and B.~Späth.
\newblock {T}he {M}c{K}ay {C}onjecture on character degrees.
\newblock {\em preprint}, 2024.

\bibitem{Enguehard}
M.~Enguehard.
\newblock Isom\'etries parfaites entre blocs de groupes sym\'etriques.
\newblock {\em Ast\'erisque}, (181-182):157--171, 1990.

\bibitem{Humphreys}
J.~F. Humphreys.
\newblock On certain projective modular representations of direct products.
\newblock {\em J. London Math. Soc. (2)}, 32:449--459, 1985.

\bibitem{Humphreys-Blocks}
J.~F. Humphreys.
\newblock Blocks of projective representations of the symmetric groups.
\newblock {\em J. London Math. Soc. (2)}, 32:441--452, 1986.

\bibitem{KennethRosenNrTh}
K.~Ireland and M.~Rosen.
\newblock {\em A {C}lassical {I}ntroduction to {M}odern {N}umber {T}heory},
  volume~84 of {\em Graduate Texts in Mathematics}.
\newblock Springer-Verlag, New York, second edition, 1990.

\bibitem{IMN}
I.~Isaacs, G.~Malle, and G.~Navarro.
\newblock A reduction theorem for {M}c{K}ay conjecture.
\newblock {\em Invent. Math.}, 170:33--101, 2007.

\bibitem{isaacs}
I.~M. Isaacs.
\newblock {\em Character {T}heory of {F}inite {G}roups}.
\newblock Academic Press [Harcourt Brace Jovanovich Publishers], New York,
  1976.
\newblock Pure and Applied Mathematics, No. 69.

\bibitem{James-Kerber}
G.~James and A.~Kerber.
\newblock {\em The {R}epresentation {T}heory of the {S}ymmetric {G}roup},
  volume~16 of {\em Encyclopedia of Mathematics and its Applications}.
\newblock Addison-Wesley Publishing Co., Reading, Mass., 1981.

\bibitem{Kessar-Schaps}
R.~Kessar and M.~Schaps.
\newblock Crossover {M}orita equivalences for blocks of the covering groups of
  the symmetric and alternating groups.
\newblock {\em J. Group Theory}, 9(6):715--730, 2006.

\bibitem{Navarro}
G.~Navarro.
\newblock {\em Characters and blocks of finite groups}, volume 250 of {\em
  London Mathematical Society Lecture Note Series}.
\newblock Cambridge University Press, Cambridge, 1998.

\bibitem{NavarroGalois}
G.~Navarro.
\newblock The {M}c{K}ay conjecture and {G}alois automorphisms.
\newblock {\em {A}nn. of {M}ath.}, 160:1129--1140, 2004.

\bibitem{NavarroSpaethVallejo}
G.~Navarro, B.~Sp\"{a}th, and C.~Vallejo.
\newblock A reduction theorem for the {G}alois-{M}c{K}ay conjecture.
\newblock {\em Trans. Amer. Math. Soc.}, 373(9):6157--6183, 2020.

\bibitem{olsson}
J.~B. Olsson.
\newblock {\em Combinatorics and {R}epresentations of {F}inite {G}roups}.
\newblock Vorlesungen aus dem Fachbereich Mathematik der Universit\"at GH
  Essen, Heff 20, 1993.

\bibitem{team}
L.~Ruhstorfer, A.~A.~S. Fry, B.~Sp\"ath, and J.~Taylor.
\newblock {T}owards the inductive {M}ckay-{N}avarro {C}ondition for groups of
  {L}ie type.
\newblock {\em preprint}, 2025.

\bibitem{schur}
I.~Schur.
\newblock \"{U}ber die {D}arstellung der symmetrischen und der alternierenden
  {G}ruppe durch gebrochene lineare {S}ubstitutionen.
\newblock {\em J. Reine Angew. Math.}, 139:155--250, 1911.

\end{thebibliography}
\end{document}